\documentclass[final,leqno]{siamltex}
\usepackage{style}

\usepackage[notref, notcite, color]{showkeys}

\usepackage{datetime}
\usepackage{verbatim}

\newcommand{\unit}{\texttt{1}\!\!\texttt{l}}
\newcommand{\intbar}{{- \hspace{- 1.05 em}} \int}
\newcommand{\intbr}{{- \hspace{- 0.9 em}} \int}
\def \D #1{\underline{D}_{#1}}
\def \d #1{\underline{\tilde{D}}_{#1}}
\def \H {\mathcal{H}}

\numberwithin{equation}{section}

\author{Charles M. Elliott\footnotemark[1]~ and Hans Fritz\footnotemark[1]}

\title{Time-periodic solutions of advection-diffusion equations on moving hypersurfaces}

\date{}

\begin{document}
\maketitle
\renewcommand{\thefootnote}{\fnsymbol{footnote}}
\footnotetext[1]{Mathematics Institute, Zeeman Building, University of Warwick, Coventry. CV4 7AL. UK 
(C.M.Elliott@warwick.ac.uk, H.Fritz@warwick.ac.uk)}
\begin{abstract}
In this paper we study time-periodic solutions
to advection-diffusion equations  
of a scalar quantity $u$ on a periodically moving $n$-dimensional hypersurface 
$\Gamma(t) \subset \mathbb{R}^{n+1}$. We prove existence and uniqueness of solutions
in suitable H\"older spaces. 
\end{abstract}



\section{Introduction}
In this paper we consider the advection and diffusion equation
\begin{align}
\label{advec_diff_eq}
	\Delta_{\Gamma(t)} u - u \nabla_{\Gamma(t)} \cdot v - \partial^\bullet u = f
\end{align}
for a scalar quantity 
$u: \overline{\mathcal{G}}_t \rightarrow \mathbb{R}$
on the space-time hypersurface 
$\mathcal{G}_t := \bigcup_{t \in (0,T)} \Gamma(t) \times \{t\} \subset \mathbb{R}^{n+2}$,
where $\Gamma(t)$ denotes a closed $n$-dimensional moving hypersurface 
$\Gamma(t) \subset \mathbb{R}^{n+1}$.
Such equations arise in many applications such as processes on biological cell surfaces, 
\cite{EllStiVen12} and the transport of surfactants on fluid interfaces, \cite{GarLamSti14}. 
The variational Hilbert space theory for the initial value problem has been considered in \cite{AlpEllSti14}.

Here  the motion of the hypersurface is assumed to be periodic in time with period $T$
in the sense that $\Gamma(0) = \Gamma(T)$.
The velocity, $v: {\mathcal{G}}_t \rightarrow \mathbb{R}^{n+1}$, 
of the moving hypersurface $\Gamma(t)$
is given by $v = v_{N} + v_{T}$, where $v_N$ denotes the normal velocity  of  $\Gamma(t)$
and $v_T$ is an advective velocity field tangential to  $\Gamma(t)$. 
The tangential gradient and the Laplace-Beltrami operator on $\Gamma(t)$ 
are denoted by $\nabla_{\Gamma(t)}$ and $\Delta_{\Gamma(t)}$, respectively.
The material derivative
$\partial^\bullet u$ is given by
\begin{align}
\label{material_derivative}
	\partial^\bullet u = u_t +  v_N \cdot \nabla u + v_T \cdot \nabla_{\Gamma(t)} u.
\end{align}
	A solution $u: \overline{\mathcal{G}}_t  \rightarrow \mathbb{R}$ 
of the advection-diffusion equation (\ref{advec_diff_eq}) is called periodic if
\begin{align}
	u(\cdot, 0) = u(\cdot, T) \quad \textnormal{on} \quad \Gamma(0).
	\label{periodic_condition}
\end{align}

The aim of this work is to establish the existence and uniqueness of periodic solutions. 
Since the mass $m(t):= \int_{\Gamma(t)}{u(\cdot,t)d\sigma(t)}$ 
of a solution $u$ of the advection-diffusion equation (\ref{advec_diff_eq}) 
evolves according to 
$$
	m'(t) = - \int_{\Gamma(t)}{f(\cdot,t) d\sigma(t)},
$$
a necessary condition for $u$ to be periodic is that
\begin{align}
	\int_0^T \int_{\Gamma(t)} f d\sigma(t) dt = 0.  
	\label{condition_on_f}
\end{align}
However, please note that neither $f$ nor the velocity $v$ is supposed to be periodic.
That is, in general $\lim_{t \searrow 0} f(\cdot,t) \neq \lim_{t \nearrow T} f(\cdot,t)$ 
as well as $\lim_{t  \searrow 0} v(\cdot,t) \neq \lim_{t \nearrow T} v(\cdot,t)$.
Although we prove a slightly more general result, 
the main result of this paper is that for each given mass 
$m_0 \in \mathbb{R}$
there exists a unique periodic solution $u \in \H_{2 + \alpha}(\mathcal{G}_t)$
of the advection-diffusion equation (\ref{advec_diff_eq})
with initial mass $m(0) = m_0$.
Here $\H_{2 + \alpha}(\mathcal{G}_t)$ denotes a suitable H\"older space, which we
introduce below.

Since for $f=0$ the mass $m(t)$ of a solution $u$ is constant in time, 
$u$ describes a conservative scalar quantity in this case.
In order to prove existence and uniqueness 
of periodic solutions to the advection-diffusion equation (\ref{advec_diff_eq}) 
we have to look at slightly more general linear parabolic partial 
differential equations on $\mathcal{G}_t$. In fact, we have to consider 
advection-diffusion equations of the form
\begin{align}
	\Delta_{\Gamma(t)} u - \mathfrak{c} u - \partial^\bullet u = f,
	\label{gen_advec_diff_eq}
\end{align} 
where $\mathfrak{c}: {\mathcal{G}}_t \rightarrow \mathbb{R}$ is a scalar function on the
space-time hypersurface ${\mathcal{G}}_t$ that is not required  to be periodic. 
Unfortunately, it turns out that the periodicity
condition in (\ref{periodic_condition}) is too restrictive in this case even if 
we assume (\ref{condition_on_f}), since in general the mass of a solution $u$ 
to the equation (\ref{gen_advec_diff_eq})
cannot be periodic in time, that is $m(0) \neq m(T)$. 
The reason for this is the zero-order
term in (\ref{gen_advec_diff_eq}). For example, for $f=0$ and 
$\mathfrak{c}= \nabla_{\Gamma(t)} \cdot v + \alpha$ with $\alpha \in (0, + \infty)$
the mass decays exponentially, that is $m(t) = m(0) \exp(-\alpha t)$.
Hence, a periodic solution would have to satisfy 
$m(0) = \int_{\Gamma(0)}u(\cdot, 0) d\sigma(0) = 0$ in this case, 
which is, however, a much too restrictive
assumption for applications. Therefore, we slightly relax the notion of periodicity to
\begin{align}
	u(\cdot, 0) - \intbar_{\Gamma(0)}u(\cdot,0) d\sigma(0) 
	= u(\cdot, T) - \intbar_{\Gamma(0)}u(\cdot,T) d\sigma(0)
	\quad \textnormal{on} \quad \Gamma(0),
	\label{general_periodic_condition}
\end{align} 
where $\intbr_{\Gamma} u  d\sigma := \frac{1}{|\Gamma|} \int_{\Gamma} u d\sigma$
denotes the mean value of $u$ on $\Gamma$.
Fortunately, for the advection-diffusion equation (\ref{advec_diff_eq}) this condition is equivalent
to the condition (\ref{periodic_condition}) as long as the constraint (\ref{condition_on_f}) is satisfied. 

The study of time-periodic solutions to both linear and non-linear parabolic equations
has a long history, see for example \cite{GL} and references therein,
in particular \cite{Fi}, \cite{He} and \cite{Ve},
as well as \cite{MB} and \cite{GLnonlin}. 
However, we are not aware of any previous analytic work that
studies the periodic problem on moving hypersurfaces. 
Our methodology reformulates  the problem on a moving closed hypersurface  to a problem on a fixed hypersurface 
with time varying coefficients. In order to obtain some analytic results it is then useful to 
relate the hypersurface equation to a   Neumann type 
boundary problem on a flat domain,
see  the Appendix for further explanations. 
Observe that the problem in our case is slightly more involved than the results on the oblique 
derivative problems for flat domains  in \cite{GL}, since 
we cannot assume that the zero order term $\mathfrak{c}$ is non-negative.
For example for the choice $\mathfrak{c}= \nabla_{\Gamma(t)} \cdot v$
this is certainly not true in general. 
The fact that the zero-order term $\mathfrak{c}$ can be negative is indeed the reason, 
why this paper is beyond the scope of previous works. 
Furthermore, the slightly weaker periodicity condition
(\ref{general_periodic_condition}), 
which ensures the existence of periodic solutions for any initial mass 
$m(0) = m_0 \in \mathbb{R}$, 
does not seem to have been used in the literature so far.  
Anyway, it is still possible to adopt the techniques used by Lieberman in \cite{GL}
to our problem. Therefore, our existence proof 
mainly relies on a simple fixed point iteration and standard Fredholm theory.

Our work is partially motivated by numerical simulations of periodic solutions of advection-diffusion
equations on moving $2$-dimensional surfaces for $f=0$  performed in \cite{EV}
using an evolving surface finite element method.
The periodic solutions were obtained by computing the initial value problem
for arbitrary chosen initial conditions.
Indeed, the numerical solutions for different initial conditions
with same initial mass
appear to converge very quickly to the same time-periodic solution.
We would like to emphasize that in this numerical work  the formulation (\ref{advec_diff_eq}),
which entirely avoids the use of local coordinates and surface parametrisations, 
is very  suitable for the evolving surface finite element method.

This paper is organized as follows. First, we introduce the notation and
the parabolic H\"older spaces on the space-time hypersurface $\mathcal{G}_t$
as well as their associated norms. Then we rewrite the advection-diffusion equation
on the moving hypersurface $\Gamma(t)$ as an advection-diffusion equation
with time-dependent coefficients on a fixed reference hypersurface $\mathcal{M}$.
This formulation is more amenable for our purposes, since it can be easily
related to a (non-degenerate) Neumann type boundary problem on a flat domain.
We derive this Neumann boundary problem in the Appendix.
In Section $2$, we summarize the results of this paper.
This also serves as a reader's guide to the proof of the main result in
Theorem \ref{main_result}. The detailed proofs of all results given 
in Section $2$ can be found 
in Section $3$. In the Appendix, we discuss the technique to extend a surface partial
differential equation to a non-degenerate partial differential equation on an extended
neighbourhood. This result is used in Section $2$ and $3$ to prove existence
and uniqueness to the initial boundary value problem on hypersurfaces
without using any local parametrizations of the hypersurface. 

\section{Preliminaries}
We make use of the convention to sum over repeated indices.
\subsection{Hypersurfaces}
Henceforward, we assume that $\Gamma(t) \subset \mathbb{R}^{n+1}$
is a family of closed (that is compact and without boundary), orientable, connected, 
$n$-dimensional, embedded hypersurfaces of class $C^{l}_1$, with $l \in \mathbb{N}$,
and that there is a closed, orientable, connected, 
$n$-dimensional, embedded hypersurface $\mathcal{M} \subset \mathbb{R}^{n+1}$ 
of the same class and a $C^{l}_1$-embedding 
$X: \overline{\mathcal{G}} \rightarrow \mathbb{R}^{n+2}$
from the closure of the cylinder 
$\mathcal{G} :=\mathcal{M} \times (0,T) \subset \mathbb{R}^{n+2}$
onto $\mathbb{R}^{n+2}$ such that for any $t \in [0,T]$ the map $X(\cdot,t)$
is a bijection from $\mathcal{M}$ onto $\Gamma(t)$. 
This implies that 
$\mathcal{G}_t := \bigcup_{t \in (0,T)} \Gamma(t) \times \{t\} \subset \mathbb{R}^{n+2}$ 
is an $(n+1)$-dimensional space-time hypersurface of class $C^{l}_1$. 
Here, $C^{l}_1$ refers to $l$-times continuously differentiable functions,
whose derivatives are Lipschitz.
Note that we use the notation $C^{2,1}$
for functions continuously differentiable in time and twice
continuously differentiable in space.
We also assume that the motion of $\Gamma(t)$ is periodic in time
in the sense that $\Gamma(0) = \Gamma(T)$ and $X(\cdot,0) = X(\cdot,T)$,
respectively.

\subsection{Tangential gradient and material derivative}
Let $\mathcal{N} \subset \mathbb{R}^{n+1}$ be an arbitrary hypersurface. The tangent space to $\mathcal{N}$ at 
the point $x \in \mathcal{N}$ is the linear space
$$
	T_x \mathcal{N} := 
	\left\{ \tau \in \mathbb{R}^{n+1} ~|~
	\exists \gamma \in C^{1}((-\epsilon,\epsilon), \mathbb{R}^{n+1}),~
	\gamma((-\epsilon,\epsilon)) \subset \mathcal{N},
	\gamma(0) = x, \gamma'(0) = \tau \right\},
$$
see \cite{DE_acta}.
For a function $f$ on an arbitrary hypersurface $\mathcal{N}$ differentiable at $x \in \mathcal{N}$,
we define the tangential gradient of $f$ at $x \in \mathcal{N}$ by
$$
\nabla_{\mathcal{N}} f (x)  := \nabla \hat{f} (x) - \nu(x) \nu(x) \cdot \nabla \hat{f}(x) 
	= P(x) \nabla \hat{f}(x),
	\quad 
$$ 
where $\hat{f}$ denotes a differentiable extension of $f$ to an open
neighbourhood $U \subset \mathbb{R}^{n+1}$ of $x$,
such that $\hat{f}_{|\mathcal{N} \cap U} = f_{|\mathcal{N} \cap U}$,
see \cite{DE_acta} for more details. Here, $\cdot$ denotes the
Euclidean scalar product in the ambient space, 
$\nu(x)$ is a unit normal of $\mathcal{N}$ at $x$
and $P(x) = \unit - \nu(x) \otimes \nu(x)$ is the projection 
onto the tangent space of $\mathcal{N}$ at $x$. 
The components of the tangential gradient are denoted by
$$
	\left(
	\begin{matrix}
		\D 1 f \\
		\vdots \\
		\D {n+1} f	
	\end{matrix}		
	 \right) := \nabla_{\mathcal{N}} f.
$$
For a twice continuously differentiable function $f$ on $\mathcal{N}$ 
we have the commutator rule
\begin{align}
	\D \alpha \D \beta f - \D \beta  \D \alpha f
	= \left( \mathcal{H}_{\beta \eta} \nu_\alpha 
		- \mathcal{H}_{\alpha \eta} \nu_\beta \right) \D \eta f,
		\label{commutator_rule}
\end{align}
where $\mathcal{H} := \nabla_\mathcal{N} \nu$ denotes the (extended) Weingarten map on $\mathcal{N}$.
The Laplace-Beltrami operator on $\mathcal{N}$ is defined by 
$$
	\Delta_{\mathcal{N}} f := \nabla_{\mathcal{N}} \cdot \nabla_{\mathcal{N}} f.
$$

These definitions can be easily generalized to moving hypersurfaces $\Gamma(t)$,
as well as to the space-derivatives on the space-time hypersurfaces $\mathcal{G}$
and $\mathcal{G}_t$.
Since $\mathcal{G}$ is a cylinder the definition of the time-derivative of 
a function $f$ on $\mathcal{G}$
is obvious. On the space-time hypersurface $\mathcal{G}_t$ we define the material 
derivative $\partial^\bullet f$ of a function $f: \mathcal{G}_t \rightarrow \mathbb{R}$
by
$$
	\partial^\bullet f := (f \circ X)_t \circ X^{-1}.
$$ 
Since the velocity $v$ of $\Gamma(t)$ is given by
$$
	v := X_t \circ X^{-1},
$$
this definition is consistent with formula $(\ref{material_derivative})$. 

\subsection{H\"older spaces}
For a function $f: \mathcal{G} \rightarrow \mathbb{R}$ we define the norm
\begin{align*}
& |f|_{0,\mathcal{G}} := \sup_{(x,t) \in \mathcal{G}} |f(x,t)|,
\end{align*}
and we say that $f$ is H\"older continuous in $\mathcal{G}$ 
with exponent $\alpha \in (0,1]$ if the semi-norm
\begin{align*}
& H_{\alpha,\mathcal{G}}(f) := \sup_{(x,t) \in \mathcal{G}}
	\sup_{~(y,s) \in \mathcal{G} \setminus \{(x,t)\}} 
	\frac{|f(x,t) - f(y,s)|}{|(x,t) - (y,s)|^\alpha}
\end{align*}
is finite. 
Here $|(x,t)| := \max \{ |x|, |t|^{\frac{1}{2}}\}$
with $|x| := \left( \sum_{\alpha=1}^{n+1} x_\alpha^2 \right)^{\frac{1}{2}}$
is the parabolic distance in $\mathbb{R}^{n+2}$.
Furthermore, we define the norms
\begin{align*}
& |f|_{\alpha,\mathcal{G}} := |f|_{0,\mathcal{G}} + H_{\alpha,\mathcal{G}}(f),
\\
& |f|_{1 + \alpha, \mathcal{G}} := | f |_{0, \mathcal{G}}  
	+ \langle f \rangle_{1 + \alpha,\mathcal{G}}	
	+ | \nabla_\mathcal{M} f |_{\alpha,\mathcal{G}},	
\\
& |f|_{2 + \alpha, \mathcal{G}} := | f |_{0, \mathcal{G}} 
+ | \nabla_{\mathcal{M}} f |_{0, \mathcal{G}} 
+ \langle \nabla_{\mathcal{M}} f \rangle_{1 + \alpha,\mathcal{G}}	
+ | \nabla^2_\mathcal{M} f |_{\alpha,\mathcal{G}} + | f_t |_{\alpha,\mathcal{G}}, 	
\end{align*}
where
\begin{align*}
	\langle f \rangle_{1 + \alpha,\mathcal{G}} :=  \sup_{(x,t) \in \mathcal{G}}
	\sup_{~(x,s) \in \mathcal{G} \setminus \{(x,t)\}} 
	\frac{|f(x,t) - f(x,s)|}{|t - s|^\frac{1 + \alpha}{2}}.
\end{align*}
For $k=0,1,2$, we introduce the following H\"older spaces on $\mathcal{G}$
\begin{align*}
	\H_{k + \alpha} (\mathcal{G}) :=
	\left\{ f: \mathcal{G} \rightarrow \mathbb{R} 
			~|~ |f|_{k + \alpha, \mathcal{G}} < \infty  \right\}.
\end{align*}
Obvious modifications of the above definitions lead to the definition of
 $|\cdot |_{ k+ \alpha,\mathcal{M}}$ for $k=0,1,2$.
For a function $f$ on $\mathcal{G}$ the following inequality holds
$$
	| f(\cdot,t) |_{k + \alpha, \mathcal{M}} \leq |f|_{k + \alpha, \mathcal{G}},
	\quad \forall t \in [0,T], \forall \alpha \in (0,1], k=0,1,2.
$$
Let the hypersurfaces $\mathcal{M}$ and $\Gamma(t)$ as well as the embedding $X$ 
be of class $C^{2}_1$, 
then the norm $| \cdot |_{k+\alpha, \mathcal{G}_t}$ on the linear space 
$\H_{k + \alpha}(\mathcal{G}_t) := \{ f: \mathcal{G}_t \rightarrow \mathbb{R} 
			~|~ f \circ X \in \H_{k + \alpha} (\mathcal{G}) \}$ is defined by
\begin{align*}
	| f |_{k+ \alpha, \mathcal{G}_t} := | f \circ X |_{k + \alpha, \mathcal{G}}
	\quad \textnormal{for $k = 0,1,2$.}
\end{align*}

Henceforward, $d$ denotes the oriented distance function to 
$\mathcal{M} \subset \mathbb{R}^{n+1}$, see for example \cite{DDE}.
There exists $\delta > 0$ such that the decomposition
$$
	x = a(x) + d(x) \nu(a(x)),
$$
with $a(x) \in \mathcal{M}$ is unique for all 
$x \in \mathcal{N}_\delta$, where 
\begin{align}
\label{N_delta}
 \mathcal{N}_\delta := \{ x \in \mathbb{R}^{n+1} ~|~ |d(x)| < \delta \}.
\end{align} 
Please note, that for a $C^l$-hypersurface, $l \geq 2$, the oriented distance
function $d$ is also of class $C^l$, whereas the projection 
$a: \mathcal{N}_\delta \rightarrow \mathcal{M}$ is of class
$C^{l-1}$. 
The oriented distance function $d$ is also Lipschitz continuous 
on $\mathcal{N}_\delta$ 
and so is the projection $a$, 
provided that the width $\delta$ is chosen sufficiently small and $l \geq 2$.
This can be seen as follows
\begin{align*}
	| a(x) - a(y) | &\leq | x - y | + |d(x) - d(y)| + |d(y)| | \nu(a(x)) - \nu(a(y)) |
	\\
	&\leq C| x - y| + c\delta |a(x) - a(y)|.
\end{align*} 
The extensions of the unit normal $\nu$, of the projection $P$ 
and of the Weingarten map
$\mathcal{H}$ on $\mathcal{M}$ to the neighbourhood $\mathcal{N}_\delta$ are defined
by $\nu(x) := \nabla d(x) = \nu(a(x))$, $P(x) := \unit - \nu(x) \otimes \nu(x)$ 
and by $\mathcal{H}(x):= \nabla^2 d(x)$.
For a function $f$ on $\mathcal{M}$ we define the lift to $\mathcal{N}_\delta$ by
$f^l(x):= f(a(x))$. A direct calculation yields 
\begin{align}
	& \nabla f^l = (\unit - d \mathcal{H}) (\nabla_{\mathcal{M}} f)^l,
\label{nabla_f_l}
\\	
	& (\nabla_\mathcal{M} f)^l = (\unit - d \mathcal{H})^{-1} \nabla f^l,
\label{nabla_Gamma_f_l}	
\end{align}
compare to \cite{DeD}.
The definitions of the norms $|\cdot|_{k+\alpha, \mathcal{G}_\delta}$ and of the spaces
$\H_{k+\alpha}(\mathcal{G}_\delta)$ for $k=0,1,2$ on the cylinder
$\mathcal{G}_\delta := \mathcal{N}_\delta \times (0,T)$ are obvious.
In order to prove the norm equivalence
\begin{align*}
	\frac{1}{C_{k,\alpha}} |f|_{k + \alpha, \mathcal{G}}
	\leq | f^l |_{k + \alpha, \mathcal{G}_\delta}
	\leq C_{k,\alpha} |f|_{k + \alpha, \mathcal{G}}
\end{align*}
for functions $f: \mathcal{G} \rightarrow \mathbb{R}$, we need the following statement.
\begin{lemma}
\label{lemma_norms}
For functions $f,g: \mathcal{G} \rightarrow \mathbb{R}$ the following inequalities hold
\begin{align*}
&	H_{\alpha,\mathcal{G}}(fg) \leq H_{\alpha,\mathcal{G}}(f) |g|_{0,\mathcal{G}} 
	+ |f|_{0,\mathcal{G}} H_{\alpha,\mathcal{G}}(g),
\\
&   |fg|_{\alpha,\mathcal{G}} \leq |f|_{\alpha,\mathcal{G}} |g|_{0,\mathcal{G}} + |f|_{0,\mathcal{G}} |g|_{\alpha,\mathcal{G}}, 
\end{align*}
Analogue estimates hold for the norms $| \cdot |_{k+\alpha, \mathcal{G}_\delta }$ on 
$\mathcal{G}_\delta$.
\end{lemma}
\proof
The first inequality easily follows from the definition of the H\"older coefficient.
The second inequality is then a direct result.
\hfill $\Box$

\begin{lemma}
For $\mathcal{M}$ of class $C^{3}_1$ there exist constants $C_{k,\alpha} > 0$ 
such that
\begin{align}
	\frac{1}{C_{k,\alpha}} |f|_{k + \alpha, \mathcal{G}}
	\leq | f^l |_{k + \alpha, \mathcal{G}_\delta}
	\leq C_{k,\alpha} |f|_{k + \alpha, \mathcal{G}},
	\label{norm_equivalence}
\end{align}
for $k=0,1,2$ and $\alpha \in (0,1]$.
\end{lemma}
\proof
Since the lifted function $f^l$ is constant in the normal direction,
we obtain
$$
	|f|_{0, \mathcal{G}} = |f^l|_{0, \mathcal{G}_\delta}.
$$
For the H\"older coefficient it is obvious that 
$H_{\alpha, \mathcal{G}}(f) \leq H_{\alpha, \mathcal{G}_\delta }(f^l)$.
Furthermore, since the projection $a$ is Lipschitz-continuous on $\mathcal{N}_\delta$
for $\delta > 0$ sufficiently small,
it follows that
\begin{align*}
	H_{\alpha, \mathcal{G}_\delta}(f^l)
	&=  \sup_{\substack{(x,t),(y,s) \in \mathcal{G}_\delta \\ (x,t) \neq (y,s)}} 
	\frac{|f^l(x,t) - f^l(y,s)|}{|(x,t) - (y,s)|^\alpha}
	\\
	&\leq 
	\sup_{\substack{(x,t),(y,s) \in \mathcal{G}_\delta \\ (x,t) \neq (y,s)}} 
	\frac{|(a(x),t) - (a(y),s)|^\alpha}{|(x,t) - (y,s)|^\alpha}	
	\sup_{\substack{(x,t),(y,s) \in \mathcal{G}_\delta \\ (x,t) \neq (y,s)}} 
	\frac{|f(a(x),t) - f(a(y),s)|}{|(a(x),t) - (a(y),s)|^\alpha}
	\\
	&\leq C H_{\alpha, \mathcal{G}}(f).
\end{align*}
This already establishes $|f|_{\alpha, \mathcal{G}}
	\leq | f^l |_{\alpha, \mathcal{G}_\delta}
	\leq C |f|_{\alpha, \mathcal{G}}$.
From Lemma \ref{lemma_norms} and formulas (\ref{nabla_f_l}) and (\ref{nabla_Gamma_f_l})
we then obtain
\begin{align*}
	| f |_{1 + \alpha, \mathcal{G}} 
	&= | f |_{0, \mathcal{G}} + \langle f \rangle_{1 + \alpha, \mathcal{G}}
	+ | \nabla_{\mathcal{M}} f |_{\alpha, \mathcal{G}} 
	\leq | f^l |_{0, \mathcal{G}_\delta} 
		+ \langle f^l \rangle_{1 + \alpha, \mathcal{G}_\delta}
		+ | (\nabla_{\mathcal{M}} f)^l |_{\alpha, \mathcal{G}_\delta}
\\	
	& \leq | f^l |_{0, \mathcal{G}_\delta} 
			+ \langle f^l \rangle_{1 + \alpha, \mathcal{G}_\delta}
			+ | {A}^{-1} \nabla f^l |_{\alpha, \mathcal{G}_\delta}
\\			
	& \leq C ( | f^l |_{0, \mathcal{G}_\delta} 
			+ \langle f^l \rangle_{1 + \alpha, \mathcal{G}_\delta}
			+ | \nabla f^l |_{\alpha, \mathcal{G}_\delta} )		
\end{align*}
and conversely,
\begin{align*}
	| f^l |_{1 + \alpha, \mathcal{G}_\delta}
	&= | f^l |_{0, \mathcal{G}_\delta} 
		+ \langle f^l \rangle_{1 + \alpha, \mathcal{G}_\delta} 
		+ | \nabla f^l |_{\alpha, \mathcal{G}_\delta}
	= | f^l |_{0, \mathcal{G}_\delta} 
		+ \langle f^l \rangle_{1 + \alpha, \mathcal{G}_\delta} 
		+ | {A} (\nabla_{\mathcal{M}} f )^l |_{\alpha, \mathcal{G}_\delta}
	\\
	& \leq C ( | f^l |_{0, \mathcal{G}_\delta} 
		+ \langle f^l \rangle_{1 + \alpha, \mathcal{G}_\delta} 
		+ | (\nabla_{\mathcal{M}} f )^l |_{\alpha, \mathcal{G}_\delta} )
	\\	
	& \leq C ( | f |_{0, \mathcal{G}} 
	+ \langle f^l \rangle_{1 + \alpha, \mathcal{G}}
	+ | \nabla_{\mathcal{M}} f |_{\alpha, \mathcal{G}} )
\end{align*}
where ${A} := \unit - d \mathcal{H}$. Similarly, the result for $k=2$ can be deduced
from
\begin{align*}
	& (\nabla^2 f^l)_{\alpha \beta} = A_{\alpha \delta} A_{\beta\gamma} (\D \delta \D \gamma f)^l
	 + \nabla_\alpha A_{\beta \gamma} (\D \gamma f)^l,
	\\
	& (\nabla_{\mathcal{M}}^2 f)^l_{\alpha \beta} 
	 = (A^{-1})^{\alpha \delta} (A^{-1})^{\beta \gamma} \nabla_\delta \nabla_\gamma f^l
	 + (A^{-1})^{\alpha \delta} \nabla_\delta (A^{-1})^{\beta \gamma} \nabla_\gamma f^l,
\end{align*}
which follows from formulas (\ref{nabla_f_l}) and (\ref{nabla_Gamma_f_l}).
For $| \cdot |_{2 + \alpha, \mathcal{G}}$ we then obtain
\begin{align*}
 |f|_{2 + \alpha, \mathcal{G}} &= |f|_{0,\mathcal{G}}
  + | \nabla_{\mathcal{M}} f |_{0,\mathcal{G}} 
  + \langle \nabla_{\mathcal{M}} f \rangle_{1 + \alpha,\mathcal{G}}
  + |\nabla^2_{\mathcal{M}} f |_{\alpha, \mathcal{G}} + | f_t |_{\alpha, \mathcal{G}}
  \\
  &\leq C (|f^l|_{0,\mathcal{G}_\delta}
  + | (\nabla_{\mathcal{M}} f)^l |_{0,\mathcal{G}_\delta} 
  + \langle (\nabla_{\mathcal{M}} f)^l \rangle_{1 + \alpha,\mathcal{G}_\delta}
  + |(\nabla^2_{\mathcal{M}} f)^l |_{\alpha, \mathcal{G}_\delta} 
  + | f^l_t |_{\alpha, \mathcal{G}_\delta})
  \\
  &\leq C (|f^l|_{0,\mathcal{G}_\delta}
  + | A^{-1} \nabla f^l |_{0,\mathcal{G}_\delta} 
  + \langle A^{-1} \nabla f^l \rangle_{1 + \alpha,\mathcal{G}_\delta}
\\ 
& \qquad \quad 
  + | ( (A^{-1})^{\alpha \delta} (A^{-1})^{\beta \gamma} \nabla_\delta \nabla_\gamma f^l ) |_{\alpha, \mathcal{G}_\delta} 
\\
& \qquad \quad
  + | ( (A^{-1})^{\alpha \delta} \nabla_\delta (A^{-1})^{\beta \gamma} \nabla_\gamma f^l ) |_{\alpha, \mathcal{G}_\delta}
  + | f^l_t |_{\alpha, \mathcal{G}_\delta})
  \\
  &\leq C (|f^l|_{0,\mathcal{G}_\delta}
  + | \nabla f^l |_{0,\mathcal{G}_\delta} 
  + \langle \nabla f^l \rangle_{1 + \alpha,\mathcal{G}_\delta}
  + | \nabla^2 f^l |_{\alpha, \mathcal{G}_\delta} 
  + | \nabla f^l |_{\alpha, \mathcal{G}_\delta}
  + | f^l_t |_{\alpha, \mathcal{G}_\delta}).
\end{align*}
Since 
$H_{\alpha, \mathcal{G}_\delta}(\nabla f^l) 
\leq C (|\nabla^2 f^l|_{\alpha, \mathcal{G}_\delta}
 + \langle \nabla f^l \rangle_{1 + \alpha, \mathcal{G}_\delta})$,
we can conclude that
$|f|_{2 + \alpha, \mathcal{G}} \leq C | f^l |_{2 + \alpha, \mathcal{G}_\delta}$.
The opposite direction follows in the same way. 
\hfill $\Box$

\begin{lemma}
The spaces $\H_{k+ \alpha} (\mathcal{G})$, with $k \in \{0,1,2\}$ and $\alpha \in (0,1]$,
are Banach spaces.
\end{lemma}
\proof
The statement follows directly from the equivalence (\ref{norm_equivalence}) 
and the fact that
the spaces $\H_{k + \alpha, \mathcal{G}_\delta}$, $k =0,1,2$, are Banach spaces. 
\hfill $\Box$

\begin{lemma}
\label{interpolation_estimate_lem}
Let $\alpha \in (0,1)$. For each $\epsilon > 0$ there is a constant $C(\epsilon)$ such that
\begin{align}
\label{interpolation_estimate}
	|f|_{\alpha,\mathcal{G}} \leq  c \epsilon^{1 - \alpha} |f|_{2 + \alpha,\mathcal{G}} 
		+ C(\epsilon) |f|_{0,\mathcal{G}},
	\quad \forall f \in \H_{2+ \alpha}(\mathcal{G}),
\end{align}
where the constant $c$ does not depend on $\epsilon$.
\end{lemma}
\proof
Without loss of generality we can assume that $\epsilon < \min\{\frac{\delta}{2},1\}$.
In case that $|(x,t) - (y,s)| \geq \epsilon$, we directly obtain
$$
	\frac{|f(x,t) - f(y,s)|}{|(x,t) - (y,s)|^\alpha} 
		\leq \frac{2 |f|_{0,\mathcal{G}}}{\epsilon^\alpha}.
$$
For $|(x,t) - (y,s)| < \epsilon$, we have
$$
  \lambda (x,t) + (1 - \lambda) (y,s) \in \mathcal{N}_{\frac{\delta}{2}}
  		\times [0,T],
  		\quad \forall \lambda \in [0,1],
$$
and there is some $(\xi, \chi) = \lambda^\ast (x,t) + (1 - \lambda^\ast) (y,s)$
with $\lambda^\ast \in [0,1]$ such that
\begin{align*}
	| f(x,t) - f(y,s)| &= |f^l(x,t) - f^l(y,s) |
	= |\nabla f^l(\xi,\chi) \cdot (x-y) + f^l_t(\xi,\chi) (t-s)|
	\\
	& \leq |\nabla f^l|_{0,\mathcal{G}_\delta} | x - y | 
		+ |f^l_t|_{0,\mathcal{G}_\delta} | t - s |
	\\	
	&  \leq |\nabla f^l|_{0,\mathcal{G}_\delta} | x - y | 
		+ |f^l_t|_{0,\mathcal{G}_\delta} | t - s |^{\frac{1}{2}}
	\\	
	& \leq (|\nabla f^l|_{0,\mathcal{G}_\delta} + |f^l_t|_{0,\mathcal{G}_\delta}) |(x,t) - (y,s)|
	\\
	& \leq |f^l|_{2+ \alpha, \mathcal{G}_\delta} |(x,t) - (y,s)|
	\\
	& \leq c |f|_{2+ \alpha, \mathcal{G}} |(x,t) - (y,s)|,
\end{align*}
where we have used (\ref{norm_equivalence}) in the last step.
Hence, we obtain
\begin{align*}
	\frac{|f(x,t) - f(y,s)|}{|(x,t) - (y,s)|^\alpha}  
	\leq c |f|_{2+ \alpha, \mathcal{G}} |(x,t) - (y,s)|^{1 - \alpha}
	\leq c \epsilon^{1 - \alpha} |f|_{2+ \alpha, \mathcal{G}}.  
\end{align*}
\hfill $\Box$

\subsection{Reformulation on a stationary hypersurface}
In   Lemma \ref{advection_diffusion_eq_on_reference_domain} of this subsection we reformulate the equation on the 
moving space-time hypersurface $\mathcal{G}_{t}$ to the fixed
space-time cylinder $\mathcal{G}$. 
In order to do this we introduce a time-dependent, symmetric and positive definite map
$G: \overline{\mathcal{G}} \rightarrow \mathbb{R}^{(n+1) \times (n+1)}$  defined by
\begin{align}
\label{G_defi}
	G:= (G_{\alpha\beta})_{\alpha,\beta=1,\ldots,n+1} \quad \textnormal{with}
	\quad G_{\alpha\beta} := \D \alpha X \cdot \D \beta X + \nu_\alpha \nu_\beta.
\end{align}
Since $X$ is time-periodic, we have $G(0)=G(T)$.
We write $G^{\alpha\beta}$ for the components of the inverse $G^{-1}$. We explain in subsection \ref{arbmetric}  below how this method can be expanded to arbitrary inner products
on the tangent space $T_a \mathcal{M}$, or more precisely to 
arbitrary Riemannian metrics on $\mathcal{M}$.

This definition (\ref{G_defi}) is motivated by the following observations. For $a \in \mathcal{M}$ the map 
$\hat  G:=((\nabla_\mathcal{M} X)^T \nabla_\mathcal{M} X)(a): T_a\mathcal{M} \rightarrow T_a \mathcal{M}$
is a bijective linear map on $T_a \mathcal{M}$, because $X$ is an embedding.
By adding a term $(\nu \otimes \nu)(a)$ to $\hat G$ in the definition of $G(a)$ this map $\hat G$
is extended to a bijective map on $\mathbb{R}^{n+1}$.

Henceforward, the volume form $d\sigma(t)$ on $\Gamma(t)$ is given by the 
$n$-dimensional Hausdorff measure, whereas the volume form $do(t)$
on $\mathcal{M}$ is the corresponding volume form weighted by the density $\sqrt{\det G(t)}$.
That is $do(t) := \sqrt{\det G(t)} d\sigma$, where $d\sigma$
is the volume form on $\mathcal{M}$ induced by the 
$n$-dimensional Hausdorff measure.
We also use the notation $do_0$ instead of $do(0)$.
Because of the periodicity we have $do_0 = do(T)$.
From Jacobi's formula we immediately obtain that
$\frac{d}{dt} \sqrt{\det G(t)} 
= \frac{1}{2} \sqrt{\det G(t)} G^{\alpha\beta} G_{t \alpha \beta}(t)$
and hence
\begin{align}
\label{transport_formula_metrics}
 \frac{d}{dt} \int_{\mathcal{M}}{ \tilde{f}(\cdot,t) do(t)}
 = \int_{\mathcal{M}}{ \tilde{f}_t(\cdot,t) + \frac{1}{2} G^{\alpha\beta} G_{t \alpha \beta} \tilde{f}(\cdot,t)
 do(t)}.
\end{align}
Below, it will become clear, how this formula is related to the transport formula
on the moving hypersurface $\Gamma(t)$, see \cite{DE_acta},
\begin{align}
\label{transport_formula}
 \frac{d}{dt} \int_{\Gamma(t)}{ f(\cdot,t) d\sigma(t)}
 = \int_{\Gamma(t)}{ \partial^\bullet f(\cdot,t) +  (\nabla_{\Gamma(t)} \cdot v) f(\cdot,t)
 d\sigma(t)}.
\end{align}
\begin{lemma}
\label{advection_diffusion_eq_on_reference_domain}
Let $c \in \mathbb{R}$. Suppose $\mathfrak{c},f: \mathcal{G}_t \rightarrow \mathbb{R}$
and $\tilde{\mathfrak{c}},\tilde{f}: \mathcal{G} \rightarrow \mathbb{R}$ are such that
$\tilde{\mathfrak{c}} = \mathfrak{c} \circ X$ and $\tilde{f} = f \circ X$.
Then a function $u \in \H_{2 + \alpha}(\mathcal{G}_t)$ is a solution of
\begin{align*}
	& \Delta_{\Gamma(t)} u - \mathfrak{c} u - \partial^\bullet u = f
	\quad \textnormal{in} \quad \mathcal{G}_t,
	\\ 
	& \intbar_{\Gamma(0)} u(\cdot,0) d\sigma(0) = c,	
	\\ 
	& u(\cdot, 0) = u(\cdot,T) - \left( \intbar_{\Gamma(0)} u(\cdot,T) d\sigma(0) - c \right)
	\quad \textnormal{on} \quad \Gamma(0),
\end{align*}
if and only if $\tilde{u} = u \circ X \in \H_{2 + \alpha}(\mathcal{G})$ is a solution of
\begin{align*}
	& \Delta_{g(t)} \tilde{u} - \tilde{\mathfrak{c}} \tilde{u} - \tilde{u}_t = \tilde{f}
	\quad \textnormal{in} \quad \mathcal{G},
	\\ 
	& \intbar_{\mathcal{M}} \tilde{u}(\cdot,0) do_0 = c,	
	\\ 
	& \tilde{u}(\cdot, 0) = \tilde{u}(\cdot,T) - \left( \intbar_{\mathcal{M}} \tilde{u}(\cdot,T) do_0 - c \right)
	\quad \textnormal{on} \quad \mathcal{M},
\end{align*}
where the linear elliptic operator $\Delta_{g(t)}$ is given by
\begin{align}
\label{defi_laplacian_g}
	\Delta_{g(t)} \tilde{u} := \D \alpha \left( G^{\alpha \beta} \D \beta \tilde{u} \right)
	+ \frac{1}{2} P_{\alpha \gamma} G^{\gamma \eta} G^{\beta \rho}
	  \D \beta G_{\alpha \eta} \D \rho \tilde{u} 
\end{align}
\end{lemma}
\proof 
From the definition of the material derivative we have 
$ \tilde{u}_t = ( \partial^\bullet u ) \circ X$.
Now, let $\D \alpha' u$, $\alpha = 1, \ldots, n+1$, denote the components of the tangential
gradient $\nabla_{\Gamma(t)} u$. Then we have
\begin{align}
	\label{nabla_u_tilde}
	\D \beta \tilde{u} = \D \beta X_\eta (\D \eta'u) \circ X.
\end{align}
The projection $P'$ onto the tangent space of $\Gamma(t)$ satisfies
\begin{align}
	\label{P_prime}
	P'_{\rho \eta} \circ X = G^{\alpha \beta} \D \alpha X_\rho \D \beta X_\eta.
\end{align}
We thus obtain
\begin{align*}
	\Delta_{g(t)} \tilde{u} &=  G^{\alpha \beta} \D \alpha \D \beta \tilde{u} 
	-  G^{\alpha \gamma} G^{\beta \delta} \D \alpha G_{\gamma\delta} \D \beta \tilde{u} 
	+ \frac{1}{2} P_{\alpha \gamma} G^{\gamma \eta} G^{\beta \rho}
	  \D \beta G_{\alpha \eta} \D \rho \tilde{u}
	 \\
	 &=  G^{\alpha \beta} (\D \alpha X_\rho \D \beta X_\eta)(\D \rho' \D \eta' u)\circ X
	 + G^{\alpha \beta} \D \alpha \D \beta X_\eta (\D \eta' u) \circ X
	 \\
	 & \quad 
	 -  G^{\alpha \gamma} G^{\beta \delta} \D \alpha G_{\gamma\delta}
	 	 \D \beta X_\eta ( \D \eta' {u}) \circ X 
   		+ \frac{1}{2} P_{\alpha \gamma} G^{\gamma \eta} G^{\beta \rho}
		\D \beta G_{\alpha \eta} \D \rho X_\kappa (\D \kappa' {u}) \circ X
	\\
	&=  P'_{\rho \eta} \circ X (\D \rho' \D \eta' u)\circ X
	 + G^{\alpha \beta} \D \alpha \D \beta X_\eta (\D \eta' u) \circ X
	 \\
	 & \quad 
	 -  G^{\alpha \gamma} G^{\beta \delta} 
	 	\D \alpha \D \gamma X_{\rho} \D \delta X_{\rho} 
	 	 \D \beta X_\eta ( \D \eta' {u}) \circ X 
	 \\
	 & \quad 	 
	 -  G^{\alpha \gamma} G^{\beta \delta} 
	 	\D \gamma X_{\rho} \D \alpha \D \delta X_{\rho} 
	 	 \D \beta X_\eta ( \D \eta' {u}) \circ X 
	 \\
	 & \quad	
   	+ \frac{1}{2} P_{\alpha \gamma} G^{\gamma \eta} G^{\beta \rho}
	( \D \beta \D \alpha X_\iota \D \eta X_\iota  
	 +  \D \alpha X_\iota \D \beta \D \eta X_\iota ) 
	 \D \rho X_\kappa (\D \kappa' {u}) \circ X	
	 \\
	&= \Delta_{\Gamma(t)} u\circ X
	 + G^{\alpha \beta} \D \alpha \D \beta X_\eta (\D \eta' u) \circ X
	 -  G^{\alpha \gamma} 
	 	\D \alpha \D \gamma X_{\rho}  ( \D \rho' {u}) \circ X 
	  \\
	 & \quad 	
	 -  G^{\alpha \gamma} G^{\beta \delta} 
	 	\D \alpha \D \delta X_{\rho}  \D \gamma X_{\rho} 
	 	 \D \beta X_\eta ( \D \eta' {u}) \circ X 
	 \\
	 & \quad	
   	+  P_{\alpha \gamma} G^{\gamma \eta} G^{\beta \rho}
	  \D \beta \D \alpha X_\iota \D \eta X_\iota  
	 \D \rho X_\kappa (\D \kappa' {u}) \circ X	
	  \\
	&= \Delta_{\Gamma(t)} u\circ X
   	+  P_{\alpha \gamma} G^{\gamma \eta} G^{\beta \rho}
	  ( \D \beta \D \alpha X_\iota -  \D \alpha \D \beta X_\iota ) \D \eta X_\iota  
	 \D \rho X_\kappa (\D \kappa' {u}) \circ X,	
\end{align*}
where we have used the fact that $G^{-1} \nu = \nu$ and
that $P G^{-1}$ is symmetric. From the commutator rule (\ref{commutator_rule}) 
it then directly follows that
$$
	\Delta_{g(t)} \tilde{u} = \Delta_{\Gamma(t)} u\circ X.
$$
The rest of the proof is obvious.
\hfill $\Box$

Due to the result in Lemma \ref{advection_diffusion_eq_on_reference_domain}
it is sufficient to consider the periodic problem on the reference cylinder
$\mathcal{G}$.
In particular, if $\mathfrak{c} = \nabla_{\Gamma(t)} \cdot v$ then 
$\tilde{\mathfrak{c}}$ on $\mathcal{M}$ is given by 
$\tilde{\mathfrak{c}} = \frac{1}{2} G^{\alpha \beta} G_{t \alpha \beta}
=: \frac{1}{2} \mbox{tr}_{g(t)}(g_t)$.
This can be seen as follows
\begin{align*}
	\frac{1}{2} G^{\alpha\beta} G_{t \alpha \beta}
	&= \frac{1}{2} G^{\alpha\beta} \left( \D \alpha X_{t\eta} \D \beta X_\eta
						+ \D \alpha X_{\eta} \D \beta X_{t \eta} \right)
\\
	&=	\frac{1}{2} G^{\alpha\beta} 
	 \left(\D \alpha X_\kappa (\D \kappa' v_\eta) \circ X \D \beta X_\eta
		+ \D \alpha X_{\eta} \D \beta X_\kappa (\D \kappa' v_\eta) \circ X \right)			\\
	&= 	P'_{\kappa\eta} \circ X 	(\D \kappa' v_\eta) \circ X 
	= (\nabla_{\Gamma(t)} \cdot v) \circ X,	
\end{align*}
where we have used (\ref{nabla_u_tilde}) with $\tilde{u} = X_{t \eta}$
and $u = X_{t \eta} \circ X^{-1} = v_\eta$
as well as (\ref{P_prime}).
This also shows the connection between formula (\ref{transport_formula_metrics})
and the transport formula (\ref{transport_formula}).

\subsubsection{An arbitrary Riemannian metric}\label{arbmetric}
The main results of this paper in  Section \ref{Results} are also valid for 
$G: \overline{\mathcal{G}} \rightarrow \mathbb{R}^{(n+1) \times (n+1)}$ 
defined by
\begin{align}
\label{defi_extended_metric}
 &G(a,t) X \cdot Y := g(a,t)(P(a)X,P(a)Y) + (\nu(a) \cdot X)
 (\nu(a) \cdot Y),
 \\
 & \forall X,Y \in \mathbb{R}^{n+1}, \forall (a,t) \in \mathcal{M}\times [0,T],
 \nonumber
\end{align}
where $g(t)$ is an arbitrary (sufficiently smooth) 
time-dependent Riemannian metric on $\mathcal{M}$, see also \cite{HF} for
further details.
Indeed, $G(t)$ is a kind of Cartesian representation of the metric $g(t)$.
In particular, if we choose $g(t) := X^\ast h$ to be the (periodic) pull-back metric 
of the Riemannian metric $h$ on $\Gamma(t) \subset \mathbb{R}^{n+1}$ that is induced
by the Euclidean metric in $\mathbb{R}^{n+1}$,
then definition (\ref{defi_extended_metric}) coincides with definition (\ref{G_defi}).
Using integration by parts on closed hypersurfaces, see \cite{DE_acta},
leads to Green's formula for the operator $\Delta_{g(t)}$
\begin{align*}
	&\int_{\mathcal{M}}{ G^{-1}(t) \nabla_\mathcal{M} u  \cdot \nabla_{\mathcal{M}} w do(t)}
	= \int_{\mathcal{M}}{G^{\alpha \beta} \D \alpha u \D \beta w \sqrt{\det G} d\sigma}
\\	
	& \quad = - \int_{\mathcal{M}}{ u \D \alpha ( G^{\alpha \beta} \D \beta w ) \sqrt{\det G}
				+  u G^{\alpha \beta} \D \alpha \sqrt{\det G} \D \beta w d\sigma}
\\
	& \quad = - \int_{\mathcal{M}}{ u \D \alpha ( G^{\alpha \beta} \D \beta w ) \sqrt{\det G}
+  u \frac{1}{2} G^{\alpha \beta} G^{\rho \eta} \D \alpha G_{\rho \eta}  \D \beta w 
	\sqrt{\det G} d\sigma}				
\\
	& \quad = - \int_{\mathcal{M}}{ u \left( \D \alpha ( G^{\alpha \beta} \D \beta w )
+ \frac{1}{2} G^{\alpha \beta} G^{\rho \eta} \D \alpha G_{\rho \eta}  \D \beta w 
\right)  do(t)}	
\\
	& \quad = - \int_{\mathcal{M}}{u \Delta_{g(t)} w do(t)} ,
\end{align*}
where we have used the fact that
\begin{align*}
	G^{\rho \eta} \D \alpha G_{\rho \eta} 
	&= P_{\rho \gamma} G^{\gamma \eta} \D \alpha G_{\rho \eta}
		+ \nu_\rho \nu_\gamma G^{\gamma \eta} \D \alpha G_{\rho \eta}
\\
	&= P_{\rho \gamma} G^{\gamma \eta} \D \alpha G_{\rho \eta}
		+ \nu_\rho \nu_\eta \D \alpha G_{\rho \eta}
\\
	&= P_{\rho \gamma} G^{\gamma \eta} \D \alpha G_{\rho \eta}
		-  \D \alpha (\nu_\rho \nu_\eta)  G_{\rho \eta}			
\\		
	&= P_{\rho \gamma} G^{\gamma \eta} \D \alpha G_{\rho \eta}
		-  \D \alpha \nu_\rho \nu_\rho	
		-  \nu_\eta \D \alpha \nu_\eta 	
\\		
	&= P_{\rho \gamma} G^{\gamma \eta} \D \alpha G_{\rho \eta}.		
\end{align*}
The last identity also shows that
$\Delta_{g(t)} \tilde{u} = \D \alpha \left( G^{\alpha \beta} \D \beta \tilde{u} \right)
	+ \frac{1}{2} G^{\alpha \eta} G^{\beta \rho}
	  \D \beta G_{\alpha \eta} \D \rho \tilde{u} $.
However, we keep formula (\ref{defi_laplacian_g}),
since it is also valid, if we would replace (\ref{G_defi}) 
by $G_{\alpha\beta} := \D \alpha X \cdot \D \beta X + \lambda \nu_\alpha \nu_\beta$ with $\lambda: \mathcal{M} \rightarrow (0,+\infty)$
continuously differentiable.
It can be shown that the elliptic operator
$\Delta_{g(t)}$ defined in (\ref{defi_laplacian_g})
is the usual Laplace operator on $\mathcal{M}$ with respect to the Riemannian metric $g(t)$. 
In order to see this, assume that 
$\varphi: \Omega \subset \mathbb{R}^n \rightarrow \mathcal{M} \subset
\mathbb{R}^{n+1}$
is a local parametrization of $\mathcal{M}$ and let 
$\tilde{g}_{ij}(\theta) := \frac{\partial \varphi}{\partial \theta^i}(\theta) \cdot \frac{\partial \varphi}{\partial \theta^j}(\theta)$. Furthermore, let $\tilde{g}^{ij}$
denote the components of the inverse of $\tilde{g} = (\tilde{g}_{ij})_{i,j=1, \ldots,n}$.
Since 
$P_{\alpha\beta}(\varphi(\theta)) = \tilde{g}^{ij}(\theta) \frac{\partial \varphi_\alpha}{\partial \theta^i}(\theta) \frac{\partial \varphi_\beta}{\partial \theta^j}(\theta)$, 
we immediately obtain from definition (\ref{defi_extended_metric}) that
\begin{align}
\label{G_g}
	& G_{\alpha \beta}(\varphi(\theta),t) :=
	\frac{\partial \varphi_\alpha}{\partial \theta^i}(\theta) \tilde{g}^{ij}(\theta)
	g_{jk}(\theta,t) \tilde{g}^{kl}(\theta) 
	\frac{\partial \varphi_\beta}{\partial \theta^l}(\theta)
	+ \nu_{\alpha}(\varphi(\theta)) \nu_\beta(\varphi(\theta)),
\\	
\label{G_g_inv}
	& G^{\alpha \beta}(\varphi(\theta),t) :=
	\frac{\partial \varphi_\alpha}{\partial \theta^i}(\theta) g^{ij}(\theta,t)
	\frac{\partial \varphi_\beta}{\partial \theta^j}(\theta)
	+ \nu_{\alpha}(\varphi(\theta)) \nu_\beta(\varphi(\theta)),
\end{align}
where $g_{ij}(\theta,t) 
:= g(\varphi(\theta),t)\left(\frac{\partial \varphi}{\partial \theta^i}(\theta) , 
\frac{\partial \varphi}{\partial \theta^j}(\theta) \right)$ are the components of the metric 
with respect to local coordinates and 
$ (g^{ij}(\theta,t))_{i,j= 1, \ldots,n} = (g_{ij}(\theta,t))^{-1}_{i,j= 1, \ldots,n}$
are the components of the inverse matrix.
For a function $u: \mathcal{M} \rightarrow \mathbb{R}^{n+1}$ the tangential gradient on $\mathcal{M}$ satisfies
\begin{align*}
	(\nabla_{\mathcal{M}} u) \circ \varphi = \tilde{g}^{ij}
	  \frac{\partial U}{\partial \theta^i} \frac{\partial \varphi}{\partial \theta^j},
	  \quad \textnormal{on $\Omega$,}
\end{align*}
where $U:= u \circ \varphi$,
see for example in \cite{DDE}. Using relation (\ref{G_g_inv}), we obtain
\begin{align*}
 &\D \alpha ( G^{\alpha \beta} \D \beta u)(\varphi(\theta))
 = \tilde{g}^{ij}(\theta) \frac{\partial \varphi_\alpha}{\partial \theta^i}(\theta)
 \frac{\partial}{\partial \theta^j}
 \left( (G^{\alpha \beta} \D \beta u )(\varphi(\theta)) \right)
 \\
 &\quad = \tilde{g}^{ij}(\theta) \frac{\partial \varphi_\alpha}{\partial \theta^i}(\theta)
 \frac{\partial}{\partial \theta^j}
 \left( \frac{\partial \varphi_\alpha}{\partial \theta^m} g^{mn} \frac{\partial U}{\partial \theta^n} \right)(\theta)
 \\
 &\quad =  \frac{\partial}{\partial \theta^m}
 \left( g^{mn} \frac{\partial U}{\partial \theta^n} \right)(\theta)
 + \tilde{g}^{ij}(\theta) \frac{\partial \varphi_\alpha}{\partial \theta^i}(\theta)
 \cdot \frac{\partial^2 \varphi_\alpha}{\partial \theta^j \partial \theta^m}(\theta)
 g^{mn}(\theta) \frac{\partial U}{\partial \theta^n}(\theta).
\end{align*}
The same procedure gives
\begin{align*}
 &\left( P_{\alpha\gamma} G^{\gamma\eta} G^{\beta\rho} \D \beta G_{\alpha\eta} \D \rho u
 \right) \circ \varphi
 \\
 &\quad = \frac{\partial \varphi_\alpha}{\partial \theta^i} g^{ij} \frac{\partial \varphi_\eta}{\partial \theta^j} 
   \frac{\partial \varphi_\beta}{\partial \theta^m} g^{mn} \frac{\partial \varphi_\rho}{\partial \theta^n}
   \frac{\partial \varphi_\beta}{\partial \theta^k} \tilde{g}^{kl} 
   \frac{\partial ( G_{\alpha \eta} \circ \varphi ) }{\partial \theta^l}  
   \frac{\partial \varphi_\rho}{\partial \theta^s} \tilde{g}^{st} \frac{\partial U}{\partial \theta^t} 
   \\
 &\quad = \frac{\partial \varphi_\alpha}{\partial \theta^i} g^{ij} \frac{\partial \varphi_\eta}{\partial \theta^j} 
   g^{lt} 
   \frac{\partial ( G_{\alpha \eta} \circ \varphi ) }{\partial \theta^l}  
   \frac{\partial U}{\partial \theta^t} 
   \\
 &\quad = \frac{\partial \varphi_\alpha}{\partial \theta^i} g^{ij} \frac{\partial \varphi_\eta}{\partial \theta^j} 
   g^{lk} 
   \frac{\partial}{\partial \theta^l}
   \left(
   		\frac{\partial \varphi_\alpha}{\partial \theta^n} \tilde{g}^{nm}
   		g_{mu} \tilde{g}^{uv} \frac{\partial \varphi_\eta}{\partial \theta^v}
   		+ (\nu_\alpha \nu_\eta) \circ \varphi
   \right)  
   \frac{\partial U}{\partial \theta^k} 
   \\
 &\quad =  g^{mu} g^{lk} 
   \frac{\partial g_{mu}}{\partial \theta^l}
   \frac{\partial U}{\partial \theta^k}  
   + \frac{\partial \varphi_\alpha}{\partial \theta^i} g^{ij} \frac{\partial \varphi_\eta}{\partial \theta^j} 
   g^{lk} g_{mu}
   \frac{\partial}{\partial \theta^l}
   \left(
   		\frac{\partial \varphi_\alpha}{\partial \theta^n} \tilde{g}^{nm}
   		 \tilde{g}^{uv} \frac{\partial \varphi_\eta}{\partial \theta^v}
   \right)  
   \frac{\partial U}{\partial \theta^k}  
   \\
 &\quad = g^{mu} g^{lk} 
   \frac{\partial g_{mu}}{\partial \theta^l}
   \frac{\partial U}{\partial \theta^k}  
   + 2 \frac{\partial \varphi_\eta}{\partial \theta^u} 
   g^{lk} 
   \frac{\partial}{\partial \theta^l}
   \left(
   		 \tilde{g}^{uv} \frac{\partial \varphi_\eta}{\partial \theta^v}
   \right)  
   \frac{\partial U}{\partial \theta^k}  
   \\
 &\quad = g^{mu} g^{lk} 
   \frac{\partial g_{mu}}{\partial \theta^l}
   \frac{\partial U}{\partial \theta^k}  
   - 2 g^{lk} \tilde{g}^{uv} \frac{\partial \varphi_\eta}{\partial \theta^v}  
   \frac{\partial^2 \varphi_\eta}{\partial \theta^l \partial \theta^u} 
   \frac{\partial U}{\partial \theta^k}. 
\end{align*}
Altogether, we obtain 
\begin{align*}
&\left( \D \alpha ( G^{\alpha \beta} \D \beta u) 
	+ \frac{1}{2} 
	P_{\alpha\gamma} G^{\gamma\eta} G^{\beta\rho} \D \beta G_{\alpha\eta} \D \rho u
 \right) \circ \varphi
 \\
 &\quad = \frac{\partial}{\partial \theta^m}
 	\left( g^{mn} \frac{\partial U}{\partial \theta^n} \right)
 	+ \frac{1}{2} g^{mu} g^{lk} 
   \frac{\partial g_{mu}}{\partial \theta^l}
   \frac{\partial U}{\partial \theta^k}  
 \\
 &\quad = \frac{1}{\sqrt{\det (g_{ij})}}
   \frac{\partial}{\partial \theta^m}
   \left( 
		\sqrt{\det (g_{ij})} g^{mn} \frac{\partial U}{\partial \theta^n}  
   \right),
\end{align*}
which shows that $\Delta_{g(t)}$ is indeed the Laplace operator on $\mathcal{M}$
with respect to $g(t)$.

\subsection{Strong maximum principle}
The (strong) maximum principle for para\-bolic partial differential equations
in flat domains,
see for example in \cite{PW}, is also valid for parabolic partial differential
equations on closed hypersurfaces. 
\begin{lemma}
\label{strong_max_principle}
Suppose that the hypotheses of Theorem \ref{existence_theorem} hold
and that $\mathcal{M}$ is connected.
Furthermore, suppose that $\Delta_{g(t)} u - \mathfrak{c} u - u_t \geq 0$ in $\mathcal{G}$ and that $u(x^\ast, t^\ast) = \max_{\overline{\mathcal{G}}} u =: M$ 
for some $(x^\ast,t^\ast) \in \overline{\mathcal{G}}$ with $t^\ast >0$.
Then $u = M$ on $\mathcal{M} \times [0,t^\ast]$ if $\mathfrak{c} = 0$,
or if $\mathfrak{c} \geq 0$ and $M \geq 0$.
\end{lemma}
\begin{proof}
We use the maximum principle in flat domains by observing that
a linear parabolic operator $L$ on $\mathcal{M} \subset \mathbb{R}^{n+1}$ 
can be extended
to a linear parabolic operator $\hat{L}$ on an open strip
$\mathcal{N}_\delta \subset \mathbb{R}^{n+1}$ about $\mathcal{M}$ such that
$$
	\hat{L} u^l(x,t) = L u(a(x),t), \quad \forall (x,t) \in \mathcal{N}_\delta \times [0,T],
$$
see in the Appendix for more details. Hence, if $Lu \geq 0$, we also have 
$\hat{L} u^l \geq 0$. 
Moreover, $u(x^\ast,t^\ast) = \max_{\overline{\mathcal{G}}} u =: M$
for some $(x^\ast,t^\ast) \in \overline{\mathcal{G}}$ if and only if 
$u^l(x^\ast,t^\ast) = \max_{\overline{\mathcal{G}}_\delta} u^l$. 
From the strong maximum principle in flat domains, see for example \cite{PW},
it therefore follows that
the set $S := \{ x \in \mathcal{M} ~|~ u(x,t) = M \}$ for any fixed $t \in [0,T]$
must be open, provided that the zero-order term $\mathfrak{c}$
satisfies $\mathfrak{c}=0$ or that $\mathfrak{c} \geq 0$ and $M \geq 0$. 
Since $S$ is also closed and $\mathcal{M}$
is connected, we have either $S = \emptyset$
or $S = \mathcal{M}$.
\end{proof}
\section{Periodic solutions: Results}\label{Results}
The starting point for the study of time-periodic solutions on hypersurfaces
is the following result on the existence and uniqueness of solutions to the
corresponding initial value problem. 

\begin{theorem}
\label{existence_theorem}
Let $\mathcal{M} \subset \mathbb{R}^{n+1}$ be a closed, orientable, 
$n$-dimensional hypersurface of class
$C^{3}_1$, and let $g(t)$, $t \in [0,T]$, be a family of Riemannian metrics on $\mathcal{M}$
such that the map $G$ defined in (\ref{defi_extended_metric}) is
of class $\H_{1+\alpha}(\mathcal{G})$ for some $\alpha \in (0,1)$. 
Furthermore, let $\mathfrak{c} \in \H_\alpha (\mathcal{G})$.
Then for any $f \in \H_\alpha (\mathcal{G})$ and $u_0 \in C^0(\mathcal{M})$ there is a unique solution
of 
\[
(S_I) = \left\{
\begin{aligned}
&	\Delta_{g(t)} u - \mathfrak{c} u - u_t = f \quad 
	\textnormal{in} \quad \mathcal{M} \times (0,T),
	\\
&   u(\cdot,0) = u_0 \quad \textnormal{on} \quad \mathcal{M}.	
\end{aligned} 
\right.
\]
If $u_0 \in \H_{2 + \alpha}(\mathcal{M})$, then $u \in \H_{2 + \alpha}(\mathcal{G})$ and there is 
a constant $C$ such that 
\begin{align}
 |u|_{2+\alpha, \mathcal{G}} \leq C ( |f|_{\alpha,\mathcal{G}} + |u_0|_{2 + \alpha,\mathcal{M}} ).
 \label{a_priori_estimate_1}
\end{align}
\end{theorem}

Using this result and a fixed point argument, it is possible to prove 
the existence of periodic solutions 
for advection-diffusion equations on $\mathcal{M}$ 
with an explicit lower bound on the zero-order term $\mathfrak{c}$.

\begin{proposition}
\label{prop_fix_point_iteration}
Let $\mathcal{M} \subset \mathbb{R}^{n+1}$ be a closed, orientable and connected,
$n$-dimensional hypersurface of class $C^{3}_1$, 
and let $g(t)$, $t \in [0,T]$, be a family of Riemannian metrics on $\mathcal{M}$ 
such that the map $G$ defined in (\ref{defi_extended_metric}) is of class 
$\H_{1+ \alpha}(\mathcal{G})$ for some $\alpha \in (0,1)$. 
Furthermore, let $\mathfrak{c}, f \in \H_\alpha(\mathcal{G})$.
If $\mathfrak{c} \geq c_0 > \frac{\ln 2}{T}$, then there is a unique solution
$u \in \H_{2 + \alpha}(\mathcal{G})$ of
\[
(S_P) = \left\{
\begin{aligned}
	& \Delta_{g(t)} u - \mathfrak{c} u - {u_t} = f \quad \textnormal{in} \quad \mathcal{M} \times (0,T),
	\\
	& \intbar_{\mathcal{M}}{u(\cdot, 0) do_0} = 0,
	\\
	& u(\cdot,0) = u(\cdot, T) - \intbar_{\mathcal{M}}{u(\cdot,T)do_0},
	\quad \textnormal{on} \quad \mathcal{M}
\end{aligned}
\right.
\]
with 
\begin{align}
	|u|_{2 + \alpha, \mathcal{G}} \leq C |f|_{\alpha,\mathcal{G}},
	\label{a_priori_estimate_2}
\end{align}
for some constant $C$ depending on $\mathcal{M}$, $g$ and $\mathfrak{c}$.
\end{proposition}

This result is sufficient to establish conditional existence
for the periodic problem without a lower bound on $\mathfrak{c}$.

\begin{proposition}[Fredholm alternative]
\label{fredholm_alternative}
Suppose that the hypotheses of Proposition \ref{prop_fix_point_iteration} hold, 
then either the homogeneous problem
\begin{align*}
	& \Delta_{g(t)} u - \mathfrak{c} u - {u_t} = 0 \quad \textnormal{in} \quad \mathcal{M} \times (0,T),
	\\
	& \intbar_{\mathcal{M}}{u(\cdot,0) do_0} = 0,
	\\
	& u(\cdot, 0) = u(\cdot, T) - \intbar_\mathcal{M}{ u(\cdot,T) do_0}
	\quad \textnormal{on} \quad \mathcal{M},
\end{align*}
has zero as its only solution, in which case the problem
\begin{align*}
	& \Delta_{g(t)} u - \mathfrak{c} u - {u_t} = f \quad \textnormal{in} \quad \mathcal{M} \times (0,T),
	\\
	& \intbar_{\mathcal{M}}{u(\cdot, 0) do_0} = c,
	\\
	& u(\cdot, 0) = u(\cdot, T) - \left( \intbar_{\mathcal{M}}{u(\cdot,T) do_0} - c \right)
	\quad \textnormal{on} \quad \mathcal{M},
\end{align*}
is solvable in the class $\H_{2 + \alpha}(\mathcal{G})$ 
for all $f \in \H_\alpha(\mathcal{G})$ and $c \in \mathbb{R}$, 
or the homogeneous problem has non-zero solutions, in which case
the non-homogeneous problem cannot be solved for some choices of
$f \in \H_\alpha(\mathcal{G})$ and $c \in \mathbb{R}$. 
\end{proposition}

Using this Proposition, one can prove existence for 
the special choice $\mathfrak{c}=0$.

\begin{corollary}
\label{cor_existence}
Suppose that the hypotheses of Proposition \ref{prop_fix_point_iteration} hold.
Then for all $f \in \H_\alpha(\mathcal{G})$ and $c \in \mathbb{R}$ 
there exists a unique solution 
$u \in \H_{2 + \alpha}(\mathcal{G})$ of
\begin{align*}
	& \Delta_{g(t)} u - {u_t} = f \quad \textnormal{in} \quad \mathcal{M} \times (0,T),
	\\
	& \intbar_{\mathcal{M}}{ u(\cdot, 0) do_0} = c,
	\\
	& u(\cdot, 0) = u(\cdot, T) - \left( \intbar_{\mathcal{M}}{u(\cdot,T) do_0} - c \right)
	\quad \textnormal{on} \quad \mathcal{M}.
\end{align*}
\end{corollary}

Existence of solutions for the adjoint operator 
$L^\ast u := \Delta_{g(t)} u (x,t) + {u}_t(x,t) $ 
of the operator $L u := \Delta_{g(t)} u - \tfrac{1}{2} \mbox{tr}_g({g}_t) u - {u}_t$
can now be established in the following way.
First, we define $\underline{g}(t) := g(T-t)$ 
and $\underline{f}(\cdot,t) := f(\cdot, T-t)$
for all $ t \in [0,T]$. Here $f \in \H_{\alpha}(\mathcal{G})$ and $g(t)$
is assumed to be a given family of Riemannian metrics with $g(0) = g(T)$
and with $G(t)$, defined as in 
(\ref{defi_extended_metric}), of class $\H_{1+\alpha}(\mathcal{G})$. 
From \mbox{Corollary \ref{cor_existence}} it follows that there
exists a unique solution $\underline{u} \in \H_{2 + \alpha}(\mathcal{G})$ of
\begin{align*}
	& \Delta_{\underline{g}(t)} \underline{u} - {\underline{u}_t} = \underline{f}
	\quad \textnormal{in} \quad \mathcal{M} \times (0,T),
	\\
	& \intbar_{\mathcal{M}}{\underline{u}(\cdot, 0) do_0} = 0,
	\\
	& \underline{u}(\cdot, 0) 
	= \underline{u}(\cdot, T) - \intbar_\mathcal{M}{\underline{u}(\cdot, T) do_0}
	\quad \textnormal{on} \quad \mathcal{M}.
\end{align*}
Next, we define $u(\cdot,t) := \underline{u}(\cdot, T -t)$. Obviously, we have
${u}_t(\cdot, t) = - {\underline{u}_t}( \cdot, T-t )$.
Furthermore, it follows that
\begin{align}
	&\Delta_{g(t)} u (x,t) + {u}_t(x,t) 
	= \Delta_{\underline{g}(T-t)} \underline{u}(x,T-t) - {\underline{u}_t}(x,T-t)
	= \underline{f}(x,T-t)
	= f(x,t),
	\label{backward_heat_eq}
	\\
	&  \intbar_\mathcal{M} u( \cdot, T) do_0 = 0,
	\nonumber
	\\
	& u(\cdot,T) = u(\cdot,0) - \intbar_\mathcal{M}{u(\cdot,0) do_0}
	\quad \textnormal{on} \quad \mathcal{M}.
	\nonumber
\end{align}

We use this result below to prove uniqueness of periodic solutions for the choice 
$\mathfrak{c}= \tfrac{1}{2} \mbox{tr}_g({g}_t)$ of the zero-order term.
That is for the operator 
$L u := \Delta_{g(t)} u - \tfrac{1}{2} \mbox{tr}_g({g}_t) u - {u}_t$. 
The Fredholm alternative in Proposition \ref{fredholm_alternative}
then gives the main theorem of the paper.

\begin{theorem}
\label{main_result}
Suppose that the hypotheses of Proposition \ref{prop_fix_point_iteration} hold.
Then for all $f \in \H_\alpha(\mathcal{G})$ and $c \in \mathbb{R}$ 
there exists a unique solution 
$u \in \H_{2 + \alpha}(\mathcal{G})$ of
\begin{align*}
	& \Delta_{g(t)} u - \tfrac{1}{2} \mbox{tr}_g({g}_t) u - {u}_t = f 
	\quad \textnormal{in} \quad \mathcal{M} \times (0,T),
	\\
	& \intbar_\mathcal{M} u(\cdot,0) do_0 = c,
	\\
	& u(\cdot,0) = u(\cdot,T) - \left( \intbar_\mathcal{M} u(\cdot,T) do_0 - c \right)	
	\quad \textnormal{on} \quad \mathcal{M}.
\end{align*}
In particular, for all $c \in \mathbb{R}$ and for all $f \in \H_\alpha(\mathcal{G})$ with $\int_0^T \int_{\mathcal{M}}{f(\cdot,t) do(t)} dt = 0$ there exists a unique solution $u \in \H_{2 + \alpha}(\mathcal{G})$ of
\begin{align*}
	& \Delta_{g(t)} u - \tfrac{1}{2} \mbox{tr}_g({g}_t) u - {u}_t = f 
	\quad \textnormal{in} \quad \mathcal{M} \times (0,T),
	\\
	& \intbar_\mathcal{M} u(\cdot,0) do_0 = c,
	\\
	& u(\cdot,0) = u(\cdot,T)
	\quad \textnormal{on} \quad \mathcal{M}.
\end{align*}
\end{theorem}

Finally, we obtain the following existence and uniqueness result for time-periodic
solutions to advection-diffusion equations on moving hypersurfaces.

\begin{theorem}
\label{existence_on_moving_surfaces}
Let $\Gamma(t) \subset \mathbb{R}^{n+1}$, $t \in [0,T]$ 
be a family of closed, orientable, connected,
$n$-dimensional hypersurfaces of class $C^{3}_1$ with $\Gamma(0)= \Gamma(T)$,
such that there exist a hypersurface $\mathcal{M}$ satisfying the hypotheses of
Proposition \ref{prop_fix_point_iteration} and 
a $C^{3}_1$-embedding $X: \overline{\mathcal{G}} \rightarrow \overline{\mathcal{G}_t}$ in the 
sense of Section $2$.
Then for all $c \in \mathbb{R}$ and for all $f \in \H_\alpha(\mathcal{G}_t)$ with 
$\int_0^T \int_{\Gamma(t)}{f(\cdot,t) d\sigma(t)} dt = 0$ 
there exists a unique solution $u \in \H_{2 + \alpha}(\mathcal{G}_t)$ of
\begin{align*}
	& \Delta_{\Gamma(t)} u - u \nabla_{\Gamma(t)} \cdot v - \partial^\bullet {u} = f 
	\quad \textnormal{in} \quad \mathcal{G}_t,
	\\
	& \intbar_{\Gamma(0)} u(\cdot,0) d\sigma(0) = c,
	\\
	& u(\cdot,0) = u(\cdot,T)
	\quad \textnormal{on} \quad \Gamma(0).
\end{align*}	
\end{theorem}

\section{Periodic solutions: Proofs}\hfill

\noindent  {\bf Proof of Theorem \ref{existence_theorem}}
Again we use the fact that the advection-diffusion equation on $\mathcal{M}$
can be reformulated as a (non-degenerate) parabolic partial differential equation
on a neighbourhood $\mathcal{N}_\delta$ of $\mathcal{M}$, see the Appendix for more
details. The theorem is a consequence of  the norm equivalence (\ref{norm_equivalence}),
Theorem $5.18$ in \cite{GLbook}. More precisely, the norm equivalence ensures that the lifted data
on $\mathcal{N}_\delta$ is sufficiently smooth and bounded. Hence, there exists a unique
solution $\hat{u}$ of the Neumann boundary problem on $\mathcal{N}_\delta$
satisfying a parabolic Schauder estimate. 
A solution $u$ to the problem $(S_I)$ can be easily
constructed from this solution $\hat{u}$, see the Appendix for details.
Moreover, the solution $u$ to the problem $(S_I)$ is unique,
since it has to satisfy $u^l = \hat{u}$. Finally, the Schauder estimate for $u$ follows
from the norm equivalence (\ref{norm_equivalence}) and the corresponding estimate
for $\hat{u}$.
\hfill $\Box$

\noindent  {\bf Proof of Proposition \ref{prop_fix_point_iteration}}
We divide the proof in two steps. First, we show that there is a unique solution 
$u \in C^0(\mathcal{M} \times [0,T]) \cap C^{2,1}(\mathcal{M} \times (0,T))$ of $(S_P)$
by applying a contraction argument. 
Here, $C^{2,1}$ refers to functions that are continuously differentiable
with respect to time and twice continuously differentiable with
respect to the space coordinates.
Then in a second step, 
we choose a special series that converges against the periodic solution
with respect to the $| \cdot |_{2 + \alpha, \mathcal{G}}$-norm in order to establish
the Schauder estimate (\ref{a_priori_estimate_2}).
 
\textit{Step $i)$} We define $J_0: C^0(\mathcal{M}) \rightarrow C^0(\mathcal{M} \times [0,T]) \cap C^{2,1}(\mathcal{M} \times (0,T))$ 
by $J_0(u_0) := u$, where $u$ is the unique solution to the initial value problem $(S_I)$,
see Theorem \ref{existence_theorem}. 
Then we define $J: C^0(\mathcal{M}) \rightarrow C^0(\mathcal{M})$ by $J(u_0)(x) := J_0(u_0)(x,T) = u(x,T)$
for all $x \in \mathcal{M}$.
In the following we show that $J$ is a contraction with 
contraction constant $\theta_J < \tfrac{1}{2}$.
Let $w_1, w_2 \in C^0(\mathcal{M})$ and $w:= w_1 - w_2$, $v:= J(w_1) - J(w_2)$.
We thus have to show that
$$
	| v |_{0,\mathcal{M}} \leq \theta_J | w |_{0,\mathcal{M}}
$$
for some $0 \leq \theta_J < \tfrac{1}{2}$. 
Henceforward, the linear second order operator in $(S_P)$ is denoted by $L$, that is
$$
	L u := \Delta_{g(t)} u - \mathfrak{c} u - u_t.
$$
Let $\epsilon \in (\frac{\ln 2}{T}, c_0)$, then we obtain
\begin{align*}
	L e^{- \epsilon t} &= \epsilon e^{- \epsilon t} - \mathfrak{c} e^{- \epsilon t}
						= (\epsilon - \mathfrak{c}) e^{- \epsilon t}
						\\
					&\leq (\epsilon - c_0 ) e^{- \epsilon t} < 0.	
\end{align*}
Hence $U^\pm := \pm (J_0(w_1) - J_0(w_2)) - | w |_{0,\mathcal{M}} e^{- \epsilon t}$
satisfies the conditions
\begin{align*}
 LU^{\pm} &= \pm (L J_0(w_1) - LJ_0(w_2)) - | w |_{0,\mathcal{M}} L e^{- \epsilon t}
\\ 
 	&= \pm (f - f) - | w |_{0,\mathcal{M}} Le^{- \epsilon t}
 \\
 	&= - | w |_{0,\mathcal{M}} L e^{- \epsilon t} \geq 0 	
 	\quad \textnormal{in} \quad \mathcal{M} \times (0,T).
\end{align*}
Furthermore, we have
\begin{align}
	U^\pm(\cdot , 0) &= \pm (J_0(w_1) - J_0(w_2))(\cdot, 0) - | w |_{0,\mathcal{M}} 
	\nonumber \\
	&= \pm ( w_1 - w_2 ) - |w|_{0,\mathcal{M}}
	\nonumber \\
	&= \pm w - |w|_{0,\mathcal{M}} \leq 0 
	\quad \textnormal{in} \quad \mathcal{M}.
	\label{boundary_values}
\end{align}
Now we suppose that $M^{\pm} := \max_{(x,t) \in \mathcal{M} \times [0,T]} {U^\pm(x,t)} > 0$.
Then $M^\pm$ must be attained at a point $(x^\ast, t^\ast) \in \mathcal{M} \times (0,T]$.
It follows from the maximum principle, see Lemma \ref{strong_max_principle}, 
that $U^\pm(x^\ast,0) = M^\pm > 0$,
which contradicts (\ref{boundary_values}).
Hence $M^\pm \leq 0$, which means that
$U^\pm \leq 0$ in $\mathcal{M} \times [0,T]$.
It follows that
$\pm v \leq | w|_{0,\mathcal{M}} e^{- \epsilon T}$ on $\mathcal{M}$, and 
$$
 |v |_{0,\mathcal{M}} \leq e^{- \epsilon T} |w|_{0,\mathcal{M}}.
$$
This shows that $J$ is a contraction with constant $\theta_J := e^{ - \epsilon T}
< e^{ - \frac{\ln 2}{T} T} = \frac{1}{2}$.
Now we define 
$$\mathcal{B} := \left\{ u \in C^0(\mathcal{M}): \intbar_{\mathcal{M}}{u do_0} = 0 \right\}, $$
and the operator $K: \mathcal{B} \rightarrow C^0(\mathcal{M})$ by
$$
	K(u_0) = J(u_0) - \intbar_\mathcal{M}{J(u_0) do_0}. 
$$
In fact we have 
\begin{align*}
	\intbar_{\mathcal{M}}{K(u_0) do_0} 
	= \left( 1 - \intbar_{\mathcal{M}}{1 do_0} \right) \intbar_{\mathcal{M}}{J(u_0) do_0}
	= 0. 
\end{align*}
Hence $K: \mathcal{B} \rightarrow \mathcal{B}$. 
Obviously, $(\mathcal{B}, | \cdot|_{0,\mathcal{M}})$ is a (non-empty) Banach space.
In the following we show that $K$ is a contraction. Let $w_1, w_2 \in \mathcal{B}$.
We thus have to show that
$$
	| K(w_1) - K(w_2)|_{0,\mathcal{M}}
	\leq \theta_K | w_1 - w_2|_{0,\mathcal{M}}
$$
for some $0 \leq \theta_K < 1$. Using the fact that $J$ is a contraction
we obtain
\begin{align*}
 	| K(w_1) - K(w_2) |_{0,\mathcal{M}} 
 	&= | J(w_1) - \intbar_\mathcal{M}{J(w_1) do_0 }
 		- \left( J(w_2) - \intbar_\mathcal{M}{J(w_2) do_0} \right) |_{0,\mathcal{M}}
 	\\
 	&\leq | J(w_1) - J(w_2) |_{0,\mathcal{M}}	
 + | \intbar_\mathcal{M}{J(w_1) - J(w_2) do_0} |_{0,\mathcal{M}}
 	\\
 	&\leq \left( 1 + \intbar_{\mathcal{M}}{1 do_0} \right)
 	| J(w_1) - J(w_2) |_{0,\mathcal{M}}	
 	\\
 &\leq 2 \theta_J | w_1 - w_2 |_{0,\mathcal{M}}.		
\end{align*}
We set $\theta_K := 2 \theta_J < 1$. Since $K$ is a contraction with constant
$\theta_K$, it follows that there is unique function $u_0^\ast \in \mathcal{B}$
with $K(u_0^\ast) = u_0^\ast$, that is
$$
	J(u_0^\ast) - \intbar_{\mathcal{M}}{J(u_0^\ast) do_0} = u_0^\ast.
$$
Now let $u^\ast := J_0(u_0^\ast)$. We then have
\begin{align*}
	& Lu^\ast = f \quad \textnormal{in} \quad \mathcal{M} \times (0,T),
\\	
	& \intbar_{\mathcal{M}}{u^\ast(\cdot,0) do_0} = 0,
\\
	& u^\ast(\cdot,0) = u_0^\ast(\cdot) = J(u_0^\ast) - \intbar_{\mathcal{M}}{J(u^\ast_0) do_0}
	  = u^\ast(\cdot,T) - \intbar_{\mathcal{M}}{u^\ast(\cdot, T) do_0}.
\end{align*} 
Now suppose  
$\hat{u} \in C^0(\mathcal{M} \times [0,T]) \cap C^{2,1}(\mathcal{M} \times (0,T))$
is a solution of $(S_P)$.
Then we have $\hat{u}_0 := \hat{u}(\cdot, 0) \in \mathcal{B}$ and
$ \hat{u} = J_0(\hat{u}_0)$ as well as $\hat{u}(\cdot, T) = J(\hat{u}_0)$.
Moreover, it follows that 
\begin{align*}
	K(\hat{u}_0) &= J(\hat{u}_0) - \intbar_{\mathcal{M}}{J(\hat{u}_0) do_0}
	= \hat{u}(\cdot, T) - \intbar_{\mathcal{M}}{ \hat{u}(\cdot,T) do_0}
	= \hat{u}_0.
\end{align*}
Therefore, we have $\hat{u} = u^\ast$, which completes the first step of the proof.

\textit{Step $ii)$}
We now define $w_0 :=0$ and $w_{k+1} := K(w_k)$ for $k \in \mathbb{N}_0$.
By induction it follows from Theorem \ref{existence_theorem}
that $J_0(w_k) \in \H_{2+ \alpha}(\mathcal{G})$. Moreover, we have
\begin{align*}
	& | K(w_{k+m}) - K(w_k) |_{2 + \alpha,\mathcal{M}} 
	\\
	& \quad \leq | J(w_{k+m}) - J(w_k)|_{2+\alpha,\mathcal{M}} + | \intbar_\mathcal{M}{J(w_{k+m}) - J(w_k)do_0}|_{0,\mathcal{M}}
	\\
	& \quad \leq | J(w_{k+m}) - J(w_k)|_{2+\alpha,\mathcal{M}} + | J(w_{k+m}) - J(w_k)|_{0,\mathcal{M}}
	\\
	& \quad \leq 2 | J(w_{k+m}) - J(w_k)|_{2+\alpha,\mathcal{M}}
	\\
	& \quad \leq 2 | J_0(w_{k+m}) - J_0(w_k) |_{2 + \alpha,\mathcal{G}}. 
\end{align*}  
In the following we show that 
$| J_0(w_{k+m}) - J_0(w_k) |_{2 + \alpha,\mathcal{G}} \rightarrow 0$
for $k \rightarrow \infty$ and hence,
$ u_0^\ast = \lim_{k \rightarrow \infty} w_k \in \H_{2+\alpha}(\mathcal{M})$
as well as 
$u^\ast = J_0(u_0^\ast) \in \H_{2 + \alpha}(\mathcal{G})$ 
according to Theorem \ref{existence_theorem}. 
Furthermore, we then have $u^\ast = \lim_{k \rightarrow \infty} J_0(w_k)$.
We now choose $\zeta \in C^\infty([0,T])$ with $\zeta = 0$
on $[0, \tfrac{1}{3}T]$, $\zeta = 1$ on $[\tfrac{2}{3}T, T]$ and $\zeta' \geq 0$.
The function $\zeta(J_0(w_{k+1}) - J_0(w_k))$ then satisfies
\begin{align*}
&	L ( \zeta (J_0(w_{k+1}) - J_0(w_k))) = - \zeta' (J_0(w_{k+1}) - J_0(w_k)) 
	\quad \textnormal{in} \quad \mathcal{M} \times (0,T),
	\\
&   (\zeta (J_0(w_{k+1}) - J_0(w_k)))(\cdot,0) = 0 
	\quad \textnormal{on} \quad \mathcal{M}	
\end{align*}
From the Schauder estimate (\ref{a_priori_estimate_1}) it follows that
\begin{align*}
	|\zeta (J_0(w_{k+1}) - J_0(w_k)) |_{2+\alpha, \mathcal{G}} 
			&\leq C |\zeta'(J_0(w_{k+1}) - J_0(w_k))|_{\alpha, \mathcal{G}}
	\\
			&\leq C(\zeta) |J_0(w_{k+1}) - J_0(w_k)|_{\alpha, \mathcal{G}},
\end{align*}
and hence,
\begin{align}
	|J(w_{k+1}) - J(w_k)|_{2+\alpha, \mathcal{M}} 
	&\leq |\zeta (J_0(w_{k+1}) - J_0(w_k)) |_{2+\alpha, \mathcal{G}}	
	\nonumber
	\\
	&\leq C(\zeta) |J_0(w_{k+1}) - J_0(w_k)|_{\alpha,\mathcal{G}}.
	\label{hoelder_estimate_1}
\end{align}
Since $L(J_0(w_{k+2}) - J_0(w_{k+1}))= 0$, 
the estimate (\ref{a_priori_estimate_1}) also gives
\begin{align*}
| J_0(w_{k+2}) - J_0(w_{k+1}) |_{2 + \alpha, \mathcal{G}} &\leq C | w_{k+2} - w_{k+1} |_{2+\alpha, \mathcal{M}}
\leq C |K(w_{k+1}) - K(w_{k}) |_{2 + \alpha, \mathcal{M}}
\\
&\leq C |  J(w_{k+1}) - J(w_k) |_{2+\alpha, \mathcal{M}}
\\
&\leq C(\zeta) |J_0(w_{k+1}) - J_0(w_k)|_{\alpha,\mathcal{G}},
\end{align*}
where we have used (\ref{hoelder_estimate_1}) in the last step.
The interpolation estimate (\ref{interpolation_estimate}) then yields
\begin{align*}
  &| J_0(w_{k+2}) - J_0(w_{k+1}) |_{2+ \alpha, \mathcal{G}} 
\\  
  & \qquad \leq
  C(\zeta) \epsilon^{1 - \alpha} | J_0(w_{k+1}) - J_0(w_k) |_{2+ \alpha, \mathcal{G}}
  + C(\zeta,\epsilon) | J_0(w_{k+1}) - J_0(w_k) |_{0, \mathcal{G}}
\end{align*}
Moreover, the maximum principle gives
$$
	|J_0(w_{k+1}) - J_0(w_k)|_{0, \mathcal{G}} \leq | w_{k+1} - w_k |_{0,\mathcal{M}},
$$
and hence,
\begin{align*}
  & | J_0(w_{k+2}) - J_0(w_{k+1}) |_{2+ \alpha, \mathcal{G}} 
  \\
  & \quad \leq 
   C(\zeta) \epsilon^{1 - \alpha} | J_0(w_{k+1}) - J_0(w_k) |_{2+ \alpha, \mathcal{G}}
  + C(\zeta,\epsilon) | w_{k+1} - w_k |_{0, \mathcal{M}}.	
\end{align*}
Choosing $\epsilon >0 $ such that $C(\zeta) \epsilon^{1- \alpha} = \theta_K$ 
and setting $C^\ast := C(\zeta,\epsilon)$
for this choice of $\epsilon$ leads to
\begin{align*}
	| J_0(w_{k+2}) - J_0(w_{k+1}) |_{2+\alpha, \mathcal{G}} 
	\leq \theta_K | J_0(w_{k+1}) - J_0(w_k) |_{2 + \alpha, \mathcal{G}}
	+ C^\ast | w_{k+1} - w_k |_{0, \mathcal{M}}.
\end{align*}
Since $K$ is a contraction, we obtain
$|w_{k+j+1} - w_{k+j}|_{0,\mathcal{M}} \leq \theta_K^{k+j} |w_1 - w_0|_{0,\mathcal{M}}$
and
\begin{align*}
	&| J_0(w_{k+m+1}) - J_0(w_{k+1}) |_{2+\alpha, \mathcal{G}}  \leq
	\sum_{j=0}^{m-1} | J_0(w_{k+j+2}) - J_0(w_{k+j+1}) |_{2 + \alpha, \mathcal{G}}
	\\
	& \quad \leq
	\sum_{j=0}^{m-1} \left\{ \theta_K |J_0(w_{k+j+1}) - J_0(w_{k+j})|_{2+\alpha, \mathcal{G}} 
		+ C^\ast \theta^{k+j}_K | w_{1} - w_{0} |_{0,\mathcal{M}} \right\}
	\\	
	& \quad \leq
	\sum_{j=0}^{m-1} \left\{ \theta^{k+j+1}_K |J_0(w_1) - J_0(w_0)|_{2+\alpha, \mathcal{G}} 
		+ C^\ast (k+j+1) \theta^{k+j}_K | w_1 - w_0|_{0,\mathcal{M}} \right\}
	\\
	& \quad \leq
	 |J_0(w_1) - J_0(w_0)|_{2+\alpha, \mathcal{G}} \frac{\theta^{k+1}_K - \theta^{k+m+1}_K}{1 - \theta_K} 
	+ C^\ast | w_1 - w_0|_{0,\mathcal{M}}	 
	\frac{\partial }{\partial \theta_K}\frac{\theta^{k+1}_K - \theta^{k+m+1}_K}{1 - \theta_K}
	\\
	& \quad \leq
	|J_0(w_1) - J_0(w_0)|_{2+\alpha, \mathcal{G}} \frac{\theta^{k+1}_K}{1 - \theta_K}
	+ C^\ast | w_1 - w_0|_{0,\mathcal{M}}	 
	\frac{(k+1)\theta_K^{k}}{(1 - \theta_K)^2},
\end{align*}
which converges to $0$ for $k \rightarrow \infty$. Therefore, the periodic solution
$u^\ast = J_0(u^\ast_0)$ is in $\H_{2+\alpha}(\mathcal{G})$ and
$$
	| J_0(u^\ast_0) - J_0(w_{1}) |_{2+\alpha, \mathcal{G}}  \leq
	|J_0(w_1) - J_0(w_0)|_{2+\alpha, \mathcal{G}} \frac{\theta_K}{1 - \theta_K}
	+  | w_1 - w_0|_{0,\mathcal{M}}	 
	\frac{C^\ast}{(1 - \theta_K)^2}.
$$ 
Since $w_0 =0$ and $|w_1|_{2+\alpha,\mathcal{M}} \leq C |J_0(w_0)|_{2+\alpha,\mathcal{G}}
\leq C |f|_{\alpha,\mathcal{G}}$, as well as
\begin{align*}
	& |J_0(w_1)|_{2+\alpha, \mathcal{G}} \leq C ( |f|_{\alpha,\mathcal{G}} + |w_1|_{2+\alpha,\mathcal{M}})
	\leq C |f|_{\alpha,\mathcal{G}},
	\\
	& |J_0(w_1) - J_0(w_0)|_{2+\alpha,\mathcal{G}} \leq C |w_1 |_{2+\alpha,\mathcal{M}}
	\leq C |f|_{\alpha,\mathcal{G}},
\end{align*}
which hold because of (\ref{a_priori_estimate_1}), we finally obtain the estimate
$$
	|u^\ast|_{2 + \alpha, \mathcal{G}} \leq C |f|_{\alpha,\mathcal{G}}.
$$
\hfill $\Box$

\noindent {\bf Proof of Proposition \ref{fredholm_alternative}}
Since $u$ solves the non-homogeneous problem if and only if $\tilde{u} := u -c$ solves the
problem 
\begin{align*}
	& \Delta_{g(t)} \tilde{u} - \mathfrak{c} \tilde{u} - \tilde{u}_t = f + c \mathfrak{c}, 
	\\
	& \intbar_{\mathcal{M}}{ \tilde{u}(\cdot, 0) do_0} = 0,
	\\
	&  \tilde{u}(\cdot, 0) = \tilde{u}(\cdot,T) - \intbar_\mathcal{M}{\tilde{u}(\cdot,T) do_0},
\end{align*}
we can assume without loss of generality that $c=0$. Now, let 
$\hat{\mathcal{B}} := \{ u \in \H_{2+\alpha}(\mathcal{G}): \intbr_\mathcal{M} {u(\cdot, 0) do_0} = 0 \ \textnormal{and} \ 
u(\cdot,0) = u(\cdot, T) - \intbr_{\mathcal{M}}{ u( \cdot, T) do_0} 
\ \textnormal{on} \ \mathcal{M} \}$ and let $L_T: \hat{\mathcal{B}} \rightarrow \H_\alpha(\mathcal{G})$
be the linear second order operator defined by 
$$
	L_T u = \Delta_{g(t)} u - \frac{1}{T} u - {u_t}.
$$
According to Proposition \ref{prop_fix_point_iteration}
the operator $L_T$ is invertible and the inverse operator $L_T^{-1}$ is continuous.
Hence, we can define the operator 
$\mathcal{K}: \H_\alpha(\mathcal{G}) \rightarrow \hat{\mathcal{B}} \subset \H_{\alpha}(\mathcal{G})$ by
$\mathcal{K} u := L_T^{-1} \left( \frac{1}{T} u - \mathfrak{c} u \right)$.
Because of the Schauder estimate (\ref{a_priori_estimate_2})
and the fact that $\hat{\mathcal{B}}$ is compactly embedded in $\H_{\alpha}(\mathcal{G})$
this is a compact operator.
The equation $\Delta_{g(t)} u - \mathfrak{c} u - u_t = f$ is equivalent to 
\begin{align*}
 & L_T u + \left( \frac{1}{T} - \mathfrak{c} \right) u = f
 \\
 \Leftrightarrow \quad & u + \mathcal{K} u = L_T^{-1} f.
\end{align*}
Since $\unit + \mathcal{K}$ is a Fredholm operator, the second equation has a solution
if and only if $(\unit + \mathcal{K}) u = 0$ implies $u = 0$. 
The standard Fredholm theory therefore gives the result. 
\hfill $\Box$

\noindent {\bf Proof of Corollary \ref{cor_existence}}
Because of the Fredholm alternative in Proposition \ref{fredholm_alternative}
we only have to establish uniqueness of the homogeneous problem.
Suppose we have a solution $\varphi \in \H_{2 + \alpha}(\mathcal{G})$ of the homogeneous problem.
Because of the maximum principle, see Lemma \ref{strong_max_principle}, we then have
\begin{align*}
	\max_{\mathcal{M}} \varphi(\cdot, 0)
	& \geq \max_{\mathcal{M}} \varphi(\cdot,T)
	= \max_\mathcal{M} \left( \varphi(\cdot,0) + \intbar_\mathcal{M} {\varphi(\cdot,T) do_0} \right)
	\\
	& = \max_\mathcal{M} \varphi(\cdot,0) + \intbar_\mathcal{M} {\varphi(\cdot,T) do_0},
\end{align*}
and hence, $\intbr_\mathcal{M} {\varphi(\cdot,T) do_0} \leq 0$. 
In the same way we obtain
$$\max_\mathcal{M} \left( - \varphi(\cdot, 0) \right) \geq 
\max_\mathcal{M} \left( - \varphi(\cdot, 0) \right) - \intbar_{\mathcal{M}}{ \varphi(\cdot,T) do_0}$$
and $\intbr_\mathcal{M} {\varphi(\cdot,T) do_0} \geq 0$.
It follows that $ \intbr_\mathcal{M} {\varphi(\cdot,T) do_0} = 0$ and hence,
$ \varphi(\cdot, 0) = \varphi(\cdot,T)$ on $\mathcal{M}$.
Moreover, we know from the strong maximum principle that either $\varphi$ is constant
or $\max_{\mathcal{M}} {\varphi(\cdot,T)} < \max_{\mathcal{M}} {\varphi(\cdot, 0)}$.
Therefore, $\varphi$ has to be constant. From $\intbr_\mathcal{M} \varphi(\cdot,0) do_0 = 0$
it then follows that $\varphi = 0$.
\hfill $\Box$

\noindent {\bf Proof of Theorem \ref{main_result}}
The second statement easily follows from the first statement and the fact that
\begin{align*}
	\int_{\mathcal{M}}{u(\cdot,T) do_0} &= \int_\mathcal{M} u(\cdot,T) do(T)
	= \int_{\mathcal{M}}{u(\cdot, 0) do(0)} + \int_0^T \frac{d}{dt} \int_\mathcal{M} {u(\cdot,t) do(t)}dt
	\\
	&= c |\mathcal{M} | + \int_0^T \frac{d}{dt} \int_\mathcal{M} {u(\cdot,t) do(t)}dt
	\\
	&= c |\mathcal{M} | + \int_0^T \int_\mathcal{M} {{u}_t + \tfrac{1}{2} \mbox{tr}_g({g}_t) u  do(t)}dt
	\\
	&= c |\mathcal{M} | + \int_0^T \int_\mathcal{M} { \Delta_{g(t)} u  do(t)} 
		- \int_0^T \int_\mathcal{M} { f  do(t)} = c |\mathcal{M} |,
\end{align*}
that is $\intbr_\mathcal{M} u(\cdot,T) do_0 = c$ and hence, $u(\cdot,0) = u(\cdot,T)$.
In order to prove the first statement, we mention 
that according to Proposition \ref{fredholm_alternative}, 
it suffices to prove the uniqueness result for the homogeneous equation, 
that is $f=0$ and $c=0$. 
Let $u \in  \H_{2 + \alpha}(\mathcal{G})$ be a solution
of the homogeneous problem. As above, we obtain $u(\cdot, 0) = u(\cdot,T)$
on $\mathcal{M}$. Next, we choose $\varphi \in \H_{2 + \alpha}(\mathcal{G})$
such that 
\begin{align*}
	&\Delta_{g(t)}\varphi + {\varphi}_t = u 
	\quad \textnormal{in} \quad \mathcal{M} \times (0,T),
	\\
	&  \intbar_\mathcal{M} \varphi( \cdot, T) do_0 = 0,
	\\
	& \varphi(\cdot,T) = \varphi(\cdot,0) - \intbar_\mathcal{M}{\varphi(\cdot,0) do_0}
	\quad \textnormal{on} \quad \mathcal{M}.
\end{align*}
According to Corollary \ref{cor_existence} such a solution exists if we choose $f = u$
in (\ref{backward_heat_eq}).
We then obtain
\begin{align*}
0 &= \int_0^T \int_\mathcal{M} \left( \Delta_{g(t)} u - \tfrac{1}{2} \mbox{tr}_g(g_t) u - u_t \right) \varphi do(t) dt
\\
	&= \int_0^T \int_\mathcal{M} u ( \Delta_{g(t)} \varphi + \varphi_t ) do(t) dt 
	- \int_0^T \int_\mathcal{M}  \tfrac{1}{2} \mbox{tr}_g(g_t) u \varphi 
		+ (u\varphi)_t do(t)	dt	
\\
	&= \int_0^T \int_\mathcal{M} u ( \Delta_{g(t)} \varphi + \varphi_t ) do(t) dt 
	- \int_0^T  \frac{d}{dt} \int_\mathcal{M} u \varphi do(t)	dt	
\\  &= \int_0^T \int_\mathcal{M} |u|^2 do(t) dt 
	- \int_\mathcal{M} u(\cdot,T) \varphi(\cdot,T) do(T) 
	+ \int_\mathcal{M} u(\cdot,0) \varphi(\cdot,0) do(0) 	
\\  &=  \int_0^T \int_\mathcal{M} |u|^2 do(t) dt 
	- \int_\mathcal{M} u(\cdot,0) \varphi(\cdot,T) do(0) 
	+ \int_\mathcal{M} u(\cdot,0) \varphi(\cdot,0) do(0) 	
\\
	&=  \int_0^T \int_\mathcal{M} |u|^2 do(t) dt 
	- \int_\mathcal{M} u(\cdot,0) ( \varphi(\cdot,T) - \varphi(\cdot,0) ) do(0)
\\ 	
	&=	\int_0^T \int_\mathcal{M} |u|^2 do(t) dt 
		+ \int_\mathcal{M} { u(\cdot,0) do_0}  \intbar_\mathcal{M} \varphi(\cdot,0) do_0
	=	\int_0^T \int_\mathcal{M} |u|^2 do(t) dt 			
\end{align*}
Hence, we have $u = 0$. This completes the proof of the claim.
\hfill $\Box$

\noindent {\bf Proof of Theorem \ref{existence_on_moving_surfaces}}
The Theorem directly follows from Lemma \ref{advection_diffusion_eq_on_reference_domain}
and Theorem \ref{main_result}.
\hfill $\Box$

\section{Appendix}
In this section, we show that solving a parabolic partial differential 
equation (PDE) on the hypersurface $\mathcal{M} \subset \mathbb{R}^{n+1}$
is, in a certain sense, equivalent to solving a related parabolic
PDE on an open 
neighbourhood of $\mathcal{M} \subset \mathbb{R}^{n+1}$. 
The main advantage of this approach is that
the well-established machinery of parabolic PDEs on 
$(n+1)$-dimensional domains of 
$\mathbb{R}^{n+1}$ can be immediately applied after a suitable PDE on 
the open neighbourhood has been found.
The main task therefore remains to formulate a PDE
on an open neighbourhood of $\mathcal{M}$ from which the solution of the PDE on $\mathcal{M}$
can be extracted.

To motivate the idea, we first have a look at the
following second-order parabolic PDE on the $n$-dimensional hyperplane 
$\mathcal{M}_0:=\{ x \in \mathbb{R}^{n+1}: x_{n+1} = 0 \}$
\[
(P) =
\left\{
\begin{aligned}
& \Delta_{\mathcal{M}_0} u - u_t = f \quad \textnormal{in} \quad \mathcal{M}_0 \times (0,T),
	\\
&	u(\cdot,0) = u_0(\cdot) \quad \textnormal{on} \quad \mathcal{M}_0,
\end{aligned}
\right.\]
where $\Delta_{\mathcal{M}_0}u = \sum_{i=1}^n{u_{x_i x_i}}$
denotes the standard Laplacian on $\mathcal{M}_0$. 
Now, let $u \in  C^0(\mathcal{M}_0\times[0,T]) \cap C^{2,1}(\mathcal{M}_0\times(0,T))$
be a solution to this problem. 
Furthermore, let $\mathcal{N}_\delta := \{ x \in \mathbb{R}^{n+1}: |x_{n+1}| < \delta \}$ be an open neighbourhood of width $\delta$ around $\mathcal{M}_0$. The function
$u^l: \overline{\mathcal{N}}_\delta \times [0,T] \rightarrow \mathbb{R}$ 
defined by $u^l(x,t):= u(\underline{x},t)$ with $x = (\underline{x}, x_{n+1})$ is then a solution to the following (strongly) parabolic initial value boundary problem
\[ 
(P^l) = \left\{
\begin{aligned}
& \Delta \hat{u} - \hat{u}_t = f^l \quad \textnormal{in} \quad \mathcal{N}_\delta \times (0,T),
	\\ 
&	
 \hat{u}_{x_{n+1}} = 0  \quad \textnormal{on} \quad \partial \mathcal{N}_\delta \times (0,T), 
	\\	
& \hat{u}(\cdot, 0) = u^l_0(\cdot) \quad \textnormal{on} \quad \mathcal{N}_\delta.
\end{aligned}
\right. \]  
Here, $f^l: \mathcal{N}_\delta \times (0,T) \rightarrow \mathbb{R}$
and $u^l_0: \mathcal{N}_\delta \rightarrow \mathbb{R}$
are defined by $f^l(x,t):= f(\underline{x},t)$ and $u^l_0(x) := u_0(\underline{x})$,
respectively.
$\Delta \hat{u} = \sum_{i=0}^{n+1}{\hat{u}_{x_i x_i}}$
denotes the standard Laplacian on $\mathbb{R}^{n+1}$.
Obviously, we have 
$\Delta u^l(x,t) = \Delta_{\mathcal{M}_{0}}{u}(\underline{x},t)
+ u^l_{x_{n+1} x_{n+1}}(x,t) = \Delta_{\mathcal{M}_{0}}{u}(\underline{x},t)$
since $u^l_{x_{n+1}} = 0$ on $\mathcal{N}_\delta \times (0,T)$.
Conversely, let $\hat{u} \in C^0(\overline{\mathcal{N}}_\delta \times [0,T])
\cap C^{2,1}(\overline{\mathcal{N}_\delta} \times (0,T))$ be a solution of the lifted PDE $(P^l)$.
Then the function $u: \mathcal{M}_0 \times [0,T] \rightarrow \mathbb{R}$
defined by $u(\underline{x},t):= \tfrac{1}{2\delta}\int_{-\delta}^\delta{
\hat{u}((\underline{x},s),t)ds}$ for all $(\underline{x},t) \in \mathcal{M}_0 \times [0,T]$
is a solution to $(P)$. This easily follows from 
$f(\underline{x},t) = \tfrac{1}{2\delta} \int_{- \delta}^\delta{f^l((\underline{x},s),t)ds}$, $u_0(\underline{x}) = \tfrac{1}{2\delta} \int_{- \delta}^\delta{u^l_0(\underline{x},s)ds}$ and the fact that 
\begin{align*}
 &\tfrac{1}{2\delta} \int_{-\delta}^{\delta}{ \Delta \hat{u} ((\underline{x},s),t)
  - \hat{u}_t((\underline{x},s),t) ds} 
 \\
 & \quad = \tfrac{1}{2\delta} \sum_{i=1}^{n+1} \int_{-\delta}^{\delta}{ \hat{u}_{x_i x_i} ((\underline{x},s),t) }ds 
  - \tfrac{1}{2\delta} \int_{-\delta}^{\delta}{ \hat{u}_t((\underline{x},s),t) ds}
  \\
 & \quad = \tfrac{1}{2\delta} \sum_{i=1}^{n} 
  	\frac{\partial^2}{\partial x_i^2}\int_{-\delta}^{\delta}{ \hat{u}_{x_i x_i} ((\underline{x},s),t) }ds 
 + \tfrac{1}{2\delta} \left( \hat{u}_{x_{n+1}}((\underline{x}, \delta),t)
  - \hat{u}_{x_{n+1}}((\underline{x}, -\delta),t) \right) 
\\ 
  & \quad \qquad - \tfrac{1}{2\delta} 
  \frac{\partial}{\partial t} \int_{-\delta}^{\delta}{ \hat{u}((\underline{x},s),t) ds}
  \\
 & \quad = \sum_{i=1}^{n} u_{x_i x_i}(\underline{x},t) - u_t(\underline{x},t)
\end{align*}
for all $(\underline{x},t) \in \mathcal{M}_0 \times (0,T)$.

Now we want to use this idea for the following second-order parabolic PDE on
the closed hypersurface $\mathcal{M} \subset \mathbb{R}^{n+1}$
\[
(S) =
\left\{
\begin{aligned}
&	Lu = f 
	\quad \textnormal{in} \quad \mathcal{M} \times (0,T),
	\\
&	u(\cdot, 0) = u_0(\cdot) \quad \textnormal{on} \quad \mathcal{M},
\end{aligned}
\right. \]
where
\[
	Lu := \Delta_{g(t)} u + w \cdot \nabla_\mathcal{M} u - \mathfrak{c} u - u_t.
\]
Here $\Delta_{g(t)} u = \D \alpha \left( G^{\alpha \beta} \D \beta u \right)
	+ \frac{1}{2} P_{\alpha \gamma} G^{\gamma \eta} G^{\beta \rho}
	  \D \beta G_{\alpha \eta} \D \rho u$ 
denotes the Laplacian with respect to
a (sufficiently) smooth time-dependent Riemannian metric $g(t)$ on $\mathcal{M}$,
where the map $G(t)$ is the Cartesian representation of $g(t)$ 
as defined in (\ref{defi_extended_metric}).
$w$ is a given vector field on $\mathcal{M}$. 
The corresponding
lifted second-order parabolic PDE is then given by
\[
(S^l) = 
\left\{
\begin{aligned}
&	\hat{L} \hat{u} = f^l \quad \textnormal{in} \quad \mathcal{N}_\delta \times (0,T),
\\
&   \frac{\partial \hat{u}}{\partial \nu} = 0 \quad \textnormal{on} \quad \partial \mathcal{N}_\delta \times (0,T),
\\
& \hat{u}(\cdot,0) = u^l_0 (\cdot) \quad \textnormal{on} \quad \mathcal{N}_\delta,	
\end{aligned}
\right.
\]
where $\mathcal{N}_\delta$ is the open strip of width $\delta$ around
$\mathcal{M}$ defined in (\ref{N_delta}),
and $f^l(x,t) := f(a(x),t)$ as well as $u^l_0(x,t) := u(a(x),t)$
are the lifted data. 
In order to define an appropriate parabolic operator 
$\hat{L}$ on $\mathcal{N}_\delta$, we first introduce the parallel hypersurfaces 
$\mathcal{M}_s := \{ a + s \nu(a): a \in \mathcal{M} \}$ for $|s| < \delta$
and the bijective projections $a_s: \mathcal{M}_s \rightarrow \mathcal{M}$ defined
by $a_s := a_{| \mathcal{M}_s}$. Obviously, we have
$\mathcal{N}_\delta = \bigcup_{|s| < \delta} \mathcal{M}_s$. 
On $\mathcal{M}_s$ we introduce the rescaled tangential gradient
for differentiable functions $\hat{u}$ by
$$
	\tilde{\nabla}_{\mathcal{M}_s} \hat{u}(x_s)
	:= ( \unit - d(x_s) \mathcal{H}(x_s))^{-1} \nabla_{\mathcal{M}_s} \hat{u}(x_s),
	\quad \forall x_s \in \mathcal{M}_s, \forall s \in (- \delta, \delta).
$$
For $\delta >0 $ sufficiently small the map $(\unit - d\mathcal{H})(x_s)$
is indeed invertible, see \cite{DeD} for further details.
The Cartesian components of the rescaled tangential gradient 
$\tilde{\nabla}_{\mathcal{M}_s} \hat{u}$ are denoted by
$$
	\left(
	\begin{matrix}
		\d 1 \hat{u} \\
		\vdots \\
		\d {n+1} \hat{u}	
	\end{matrix}		
	 \right) := \tilde{\nabla}_{\mathcal{M}_s} \hat{u}.
$$
We then define the parabolic operator $\hat{L}$ by
\begin{align}
 \label{Laplacian}
 \hat{L} \hat{u} =& 
 	\d \alpha \left( G^{l\alpha \beta} \d \beta \hat{u} \right)
	+ \frac{1}{2} P_{\alpha\gamma} G^{l\gamma \eta} G^{l\beta \rho} 
		\d \beta G^l_{\alpha \eta} \d \rho \hat{u}
	+ \frac{\partial^2 \hat{u}}{\partial \nu^2}
	 + w^l \cdot \tilde{\nabla}_{\mathcal{M}_s} \hat{u}
	- \mathfrak{c}^l \hat{u} - \hat{u}_t,
\end{align}
where $G^{l \alpha\beta}$ denotes the components
of $(G^{l})^{-1}$ and $G^l_{\alpha \beta}$ denotes the components of 
$G^l(x,t):= G(a(x),t)$.
Furthermore, $w^l(x,t) := w(a(x),t)$ and $\mathfrak{c}^l(x,t) := \mathfrak{c}(a(x),t)$. 
Below we show that for $\hat{u} = u^l(x,s):= u(a(x),t)$ we have
\begin{align}
\label{hat_L_identity}
	\hat{L} u^l(x,t) = Lu(a(x),t), \quad \forall x \in \mathcal{N}_\delta.
\end{align}
Moreover, we have $\tfrac{\partial u^l}{\partial \nu} = 0$,
since $u^l(a + s \nu(a)) = u(a)$ for all $ a \in \mathcal{M}$, $|s| < \delta$.
Using these facts, it is easy to show that $u^l$ is a solution 
to $(S^l)$ if $u$ is a solution to $(S)$. 

Now suppose that
$\hat{u}: \overline{\mathcal{N}}_\delta \times [0,T] \rightarrow \mathbb{R}^{n+1}$ is a solution to $(S^l)$. In this case we define
$u: \mathcal{M} \times [0,T] \rightarrow \mathbb{R}$ by
$$
	u(a,t) := \frac{1}{2\delta} \int_{- \delta}^{\delta}{\hat{u}(a + s \nu(a),t) ds}, \quad \forall (a,t) \in \mathcal{M} \times [0,T].
$$
In order to see that $u$ satisfies $(S)$, we introduce the family of functions
$\tilde{u}_s: \mathcal{M} \times [0,T]  \rightarrow \mathbb{R}$
for $s \in [-\delta, \delta]$ defined by
\begin{equation}
	\tilde{u}_s(a,t) := \hat{u}(a + s \nu(a),t),
	\quad \forall (a,t) \in \mathcal{M} \times [0,T],
	\forall s \in [-\delta, \delta].
	\label{defi_u_s}
\end{equation}
We need this definition, because it is a priori not clear whether $\hat{u}$ is constant
in the normal direction.
Obviously, we have 
\begin{equation}
 \hat{u}_{|\mathcal{M}_s} = \tilde{u}^l_{s|\mathcal{M}_s}.
 \label{a_s_u_s}
\end{equation}
Since the tangential gradient $\tilde{\nabla}_{\mathcal{M}_s}$ only depends on
the values of $\hat{u}$ on $\mathcal{M}_s$, we obtain 
\begin{align*}
	\hat{L} \hat{u}(x_s,t) 
	&= \hat{L}\tilde{u}^l_s(x_s,t) 
	   - \frac{\partial^2 \tilde{u}^l_s}{\partial \nu^2}(x_s,t)
	   + \frac{\partial^2 \hat{u}}{\partial \nu^2}(x_s,t)	
\\	
	&=
	  L\tilde{u}_s(a(x_s),t) 
	   + \frac{\partial^2 \tilde{u}_s}{\partial s^2}(a(x_s),t),		
\quad \forall x_s \in \mathcal{M}_s, \forall s \in (-\delta,\delta),
\end{align*}
where we have used formula (\ref{hat_L_identity}), 
$\frac{\partial \tilde{u}_s^l}{\partial \nu} = 0$
and 
$\frac{\partial^2 \hat{u}}{\partial \nu^2}(x_s,t) = \frac{\partial^2 \tilde{u}_s}{\partial s^2}(a(x_s),t)$
in the last step.
It follows that for all $a \in \mathcal{M}$ and $t \in (0,T)$ we have
\begin{align*}
	f(a,t) &= \frac{1}{2\delta} \int_{- \delta}^{ \delta}{f^l(a + s \nu(a),t) ds} 
= \frac{1}{2 \delta} \int_{- \delta}^{ \delta}{\hat{L} \hat{u}( a + s\nu(a),t) ds}
\\
&= \frac{1}{2 \delta} \int_{- \delta}^{ \delta}{L \tilde{u}_s(a,t) ds}
+ \frac{1}{2 \delta} \int_{- \delta}^{ \delta}{ \frac{\partial^2 \tilde{u}_s}{\partial s^2}(a,t) ds}
\\
&= \frac{1}{2 \delta} L \int_{- \delta}^{ \delta}{\tilde{u}_s(a,t) ds}
\\
&= (L u)(a,t),
\end{align*}
where the term $\int_{-\delta}^\delta{\frac{\partial^2 \tilde{u}_s}{\partial s^2}ds}$ vanishes, because of the Neumann boundary condition in 
$(S^l)$. Hence, $u = \frac{1}{2\delta} \int_{-\delta}^\delta{\tilde{u}_s ds}$ is a solution to $(S)$ if $\hat{u}$ is a solution to the lifted problem $(S^l)$.
From the uniqueness of solutions to $(S^l)$ it follows that $\hat{u}= u^l$,
which shows that $\hat{u}$ has to be constant in the normal direction.

In the following, we prove formula (\ref{hat_L_identity})
and show that $\hat{L}$ is in fact a strongly parabolic second-order operator
on $\mathcal{N}_\delta$. 
First, it is easy to show that the tangential gradient of the projection
$a$ is given by
\begin{align*}
	\nabla_{\mathcal{M}_s} a(x_s) = P(x_s) - d(x_s) \mathcal{H}(x_s),
	\quad \forall x_s \in \mathcal{M}_s.  
\end{align*}
Using the fact that the $\nu(x)= \nu(a(x))$ and $P(x) = P(a(x))$, 
it follows that for $u^l(\cdot,t) := u(a(\cdot),t)$ 
the following identity hold
\begin{align*}
	\nabla_{\mathcal{M}_s} u^l(x_s,t) 
	&= 	
		( \unit - d(x_s) \mathcal{H}(x_s)) \nabla_{\mathcal{M}} u(a(x_s),t).
\end{align*}
Hence, $\forall (x_s,t) \in \mathcal{M}_s \times [0,T]$ we have
\begin{equation}
	\nabla_{\mathcal{M}} u(a(x_s),t) = 
	( \unit - d(x_s) \mathcal{H}(x_s))^{-1} \nabla_{\mathcal{M}_s} u^l(x_s,t)
	= \tilde{\nabla}_{\mathcal{M}_s} u^l(x_s,t),
	\label{identity_gradient}
\end{equation}
or $(\nabla_{\mathcal{M}} u)^l = \tilde{\nabla}_{\mathcal{M}_s} u^l$, respectively.
From this result and the fact that $\frac{\partial u^l}{\partial \nu} = 0$
we directly obtain that
\begin{align*}
	\hat{L} u^l (x,t) =& \d \alpha \left( G^{l\alpha \beta} \d \beta u^l \right)(x,t)
	+ \frac{1}{2} ( P_{\alpha\gamma} G^{l\gamma \eta} G^{l\beta \rho} 
		\d \beta G^l_{\alpha \eta} \d \rho u^l )(x,t)
\\		
	& + w^l(x,t) \cdot \tilde{\nabla}_{\mathcal{M}_s} u^l(x,t)
	- \mathfrak{c}^l(x,t) u^l(x,t) - u^l_t(x,t)
\\
	=& \D \alpha \left( G^{\alpha \beta} \D \beta u \right)(a(x),t)
	+ \frac{1}{2} ( P_{\alpha\gamma} G^{\gamma \eta} G^{\beta \rho} 
		\D \beta G_{\alpha \eta} \D \rho u )(a(x),t)
\\		
	& + w(a(x),t) \cdot \nabla_{\mathcal{M}} u(a(x),t)
	- \mathfrak{c}(a(x),t) u(a(x),t) - u_t(a(x),t)
\\	
	=& Lu (a(x),t).	
\end{align*}
In order to see that the operator $\hat{L}$ is strongly parabolic, we consider the second-order terms in (\ref{Laplacian}). Using the notation 
$A(x) := ( \unit - d \mathcal{H})(x)$ for all $x \in \mathcal{N}_\delta$ as well as 
\begin{align*}
& \left( A_{\alpha \beta} \right)_{\alpha, \beta=1, \ldots, n+1} := A,
\quad & \left( A^{\alpha \beta} \right)_{\alpha, \beta=1, \ldots, n+1} := A^{-1},	  	
\end{align*}
for the components of $A$ and $A^{-1}$, we obtain 
\begin{align*}
	&\d \alpha (G^{l\alpha \beta} \d \beta \hat{u})  + \frac{\partial^2 \hat{u}}{\partial \nu^2}
	= A^{\alpha \rho} P_{\rho\eta} D_\eta (G^{l\alpha \beta} A^{\beta \kappa} P_{\kappa\iota} D_\iota \hat{u})
	+ \frac{\partial^2 \hat{u}}{\partial \nu^2}
	\\
	&= A^{\alpha \rho} D_\rho (G^{l\alpha \beta} A^{\beta \iota} D_\iota \hat{u})
	- A^{\alpha \rho} \nu_\rho \nu_\eta D_\eta (G^{l\alpha \beta} A^{\beta \iota} D_\iota \hat{u})	
	- A^{\alpha \rho} D_\rho ( G^{l\alpha \beta} A^{\beta \kappa} \nu_\kappa \nu_\iota D_\iota \hat{u})
	\\
	&\quad + A^{\alpha \rho} \nu_\rho \nu_\eta D_\eta (G^{l\alpha \beta} A^{\beta \kappa}
	\nu_\kappa \nu_\iota D_\iota \hat{u})
	+ \frac{\partial^2 \hat{u}}{\partial \nu^2}
	\\
	&= A^{\alpha \rho} D_\rho (G^{l\alpha \beta} A^{\beta \iota} D_\iota \hat{u})
	- \nu_\alpha \nu_\eta D_\eta (G^{l\alpha \beta} A^{\beta \iota} D_\iota \hat{u})	
	- A^{\alpha \rho} D_\rho ( \nu_\alpha \nu_\iota D_\iota \hat{u})
	\\
	&\quad + \nu_\alpha \nu_\eta D_\eta ( \nu_\alpha \nu_\iota D_\iota \hat{u})
	+ \frac{\partial^2 \hat{u}}{\partial \nu^2}
	\\
	&= A^{\alpha \rho} D_\rho (G^{l\alpha \beta} A^{\beta \iota} D_\iota \hat{u})
	- \nu_\eta D_\eta ( \nu_\iota D_\iota \hat{u})	
	+ G^{l\alpha \beta} A^{\beta \iota} \nu_\eta D_\eta \nu_\alpha D_\iota \hat{u}	
	- \nu_\rho D_\rho ( \nu_\iota D_\iota \hat{u})
	\\
	&\quad - A^{\alpha \rho} D_\rho \nu_\alpha \nu_\iota D_\iota \hat{u}
	+ \nu_\eta D_\eta ( \nu_\iota D_\iota \hat{u}) + \nu_\alpha \nu_\eta D_\eta \nu_\alpha \nu_\iota D_\iota \hat{u}
	+ \frac{\partial^2 \hat{u}}{\partial \nu^2}
	\\
	&= A^{\alpha \rho} D_\rho (G^{l\alpha \beta} A^{\beta \iota} D_\iota \hat{u})
	 - A^{\alpha \rho} D_\rho \nu_\alpha \nu_\iota D_\iota \hat{u}
	 \\
	&= A^{\alpha \rho} D_\rho (G^{l\alpha \beta} A^{\beta \iota} D_\iota \hat{u})
	 -  ( \tilde{\nabla}_{\mathcal{M}_s} \cdot \nu) \frac{\partial}{\partial \nu} \hat{u}
	\\
	&= A^{\rho \alpha} G^{l\alpha \beta} A^{\beta \iota} D_\rho D_\iota \hat{u}
	 + A^{\alpha \rho} D_\rho (G^{l\alpha \beta} A^{\beta \iota}) D_\iota \hat{u}
	 -  ( \tilde{\nabla}_{\mathcal{M}_s} \cdot \nu) \frac{\partial}{\partial \nu} \hat{u}.
\end{align*}
Here we have used the fact that $(G^l)^{-1} \nu = \nu$ and 
$A^{-1} \nu = \nu$ as well as $\frac{\partial \nu}{\partial \nu} = 0$.
For $G = \unit$, this identity simplifies to
\begin{equation}
	\d \alpha \d \alpha \hat{u}  + \frac{\partial^2 \hat{u}}{\partial \nu^2}
	= A^{\rho \alpha} A^{\alpha \iota} D_\rho D_\iota \hat{u}
	 + A^{\alpha \rho} D_\rho A^{\alpha \iota} D_\iota \hat{u}
	 -  ( \tilde{\nabla}_{\mathcal{M}_s} \cdot \nu) \frac{\partial}{\partial \nu} \hat{u}.
\label{elliptic_part}
\end{equation}
\begin{remark}
A similar extension idea has been used only recently to develop
numerical schemes for the simulation of geometric PDEs on surfaces, see \cite{OS}. 
The model problem considered in \cite{OS} is the following elliptic problem
$$
  	\Delta_\Gamma u - \mathfrak{c} u = f \quad \textnormal{on} \quad \Gamma. 
$$
Using our notation, the authors propose the following non-degenerate extended equations
\[
\begin{aligned}
	& \nabla \cdot (\mu A^{-2} \nabla \hat{u}) - \mu \mathfrak{c}^l \hat{u} = \mu f^l 
		\quad \textnormal{in} \quad \mathcal{N}_\delta,
	\\
	& \frac{\partial \hat{u}}{\partial \nu} = 0
		\quad \textnormal{on} \quad \partial\mathcal{N}_\delta,
\end{aligned}
\]
where $\mu = \det A$. 
This problem is of course equivalent to the problem
$$
	\frac{1}{\mu} \nabla \cdot (\mu A^{-2} \nabla \hat{u}) - \mathfrak{c}^l \hat{u} = f^l 
		\quad \textnormal{in} \quad \mathcal{N}_\delta
$$
with zero Neumann boundary conditions.
A long calculation now reveals that in the case of $G=\unit$ 
and $w=0$ the resulting elliptic part of our 
parabolic operator $\hat{L}$ in (\ref{Laplacian}) indeed 
reduces to the operator
$\frac{1}{\mu} \nabla \cdot (\mu A^{-2} \nabla \hat{u}) - \mathfrak{c}^l \hat{u}$,
\begin{align*}
	&\frac{1}{\mu} \nabla \cdot (\mu A^{-2} \nabla \hat{u})
	= \frac{1}{\mu} D_\alpha (\mu A^{\alpha \beta} A^{\beta \gamma} D_\gamma \hat{u})
	\\
	&\quad = A^{\alpha \beta} A^{\beta \gamma} D_\alpha D_\gamma \hat{u}
	   + A^{\alpha \beta} D_\alpha A^{\beta \gamma} D_\gamma \hat{u}
	   + D_\alpha A^{\alpha \beta} A^{\beta \gamma} D_\gamma \hat{u}
	   + \frac{1}{\mu} D_\alpha \mu A^{\alpha \beta} A^{\beta \gamma} D_\gamma \hat{u} 
	\\   
	&\quad = A^{\alpha \beta} A^{\beta \gamma} D_\alpha D_\gamma \hat{u}
	   + A^{\alpha \beta} D_\alpha A^{\beta \gamma} D_\gamma \hat{u}
	   - A^{\kappa \iota} D_\kappa \nu_\iota \frac{\partial}{\partial \nu} \hat{u},
\end{align*}
since
\begin{align*}
	& D_\alpha A^{\alpha \beta} A^{\beta \gamma} D_\gamma \hat{u}
	  + \frac{1}{\mu} D_\alpha \mu A^{\alpha \beta} A^{\beta \gamma} D_\gamma \hat{u}
	\\  
	&\quad =  - A^{\alpha \kappa} A^{\beta \iota} D_\alpha A_{\kappa \iota} A^{\beta \gamma} 
		D_\gamma \hat{u}
+ A^{\kappa \iota} D_\alpha A_{\kappa \iota} A^{\alpha \beta} A^{\beta \gamma} D_\gamma \hat{u}    
	\\
	&\quad =  A^{\alpha \kappa} D_\alpha (d \mathcal{H}_{\kappa \iota}) 
				A^{\iota \beta} A^{\beta \gamma} D_\gamma \hat{u}
			- A^{\kappa \iota} D_\alpha (d \mathcal{H}_{\kappa \iota})
					 A^{\alpha \beta} A^{\beta \gamma} D_\gamma \hat{u}
	\\
	&\quad =  A^{\alpha \kappa} \nu_\alpha \mathcal{H}_{\kappa \iota} 
				A^{\iota \beta} A^{\beta \gamma} D_\gamma \hat{u}
			 + d A^{\alpha \kappa} D_\alpha D_\kappa D_\iota d
				A^{\iota \beta} A^{\beta \gamma} D_\gamma \hat{u}	
	\\
	& \quad \qquad			
			 - A^{\kappa \iota} \nu_\alpha \mathcal{H}_{\kappa \iota}
					 A^{\alpha \beta} A^{\beta \gamma} D_\gamma \hat{u}
			 - d A^{\kappa \iota} D_\alpha D_\kappa D_\iota d
					 A^{\alpha \beta} A^{\beta \gamma} D_\gamma \hat{u}	
	\\
	&\quad = d A^{\alpha \kappa} D_\iota D_\alpha D_\kappa d
				A^{\iota \beta} A^{\beta \gamma} D_\gamma \hat{u}				
			 - A^{\kappa \iota} \mathcal{H}_{\kappa \iota}
					 \nu_\gamma D_\gamma \hat{u}
			 - d A^{\kappa \iota} D_\alpha D_\kappa D_\iota d
					 A^{\alpha \beta} A^{\beta \gamma} D_\gamma \hat{u}	
	\\
	&\quad =  - A^{\kappa \iota} D_\kappa \nu_\iota
					 \frac{\partial}{\partial \nu} \hat{u}.
\end{align*}
From (\ref{elliptic_part}) it hence follows that
\begin{align*}
	\frac{1}{\mu} \nabla \cdot (\mu A^{-2} \nabla \hat{u}) =
	\d \alpha \d \alpha \hat{u}  + \frac{\partial^2 \hat{u}}{\partial \nu^2}.
\end{align*}
\end{remark}

\subsection*{Acknowledgements}
This paper resulted from a stay of both authors 
at the Isaac Newton Institute of Mathematical Sciences
in Cambridge, UK, as participants of the research programme
\textit{Free Boundary Problems and Related Topics}.
The authors would like to thank the institute, and the organizers of the programme 
for their invitation and their kind hospitality.
The second author was also supported by the Alexander von Humboldt Foundation, Germany,
by a Feodor Lynen Research Fellowship in collaboration with the University of Warwick, UK.

\end{document}